\documentclass[11pt,twoside]{amsart}
\usepackage{latexsym,amssymb,amsmath}
\usepackage[all]{xy}
\usepackage{tikz,pgfplots}
\usepackage{extpfeil}
\usepackage{afterpage}
\usepackage{float}

\numberwithin{equation}{section}
\def\demo{\noindent{\it Proof. }}
\newtheorem{theorem}{Theorem}[section]
\newtheorem{lemma}[theorem]{Lemma}
\newtheorem{proposition}[theorem]{Proposition}
\newtheorem{corollary}[theorem]{Corollary}

\newtheorem{problem}[theorem]{Problem}

\theoremstyle{definition}
\newtheorem{definition}[theorem]{Definition} 
\newtheorem{procedure}[theorem]{Procedure} 
\newtheorem{remark}[theorem]{Remark}
\newtheorem{example}[theorem]{Example}

\begin{document}


\title[Evaluation codes and their basic parameters]
{Evaluation codes and their basic parameters} 

\thanks{The first author was supported by a scholarship from CONACYT,
Mexico. The second author
was partially supported by the Center for Mathematical Analysis, Geometry and
Dynamical Systems of Instituto Superior T\'ecnico, Universidade de
Lisboa. The third author was supported by SNI, Mexico.}

\author[D. Jaramillo]{Delio Jaramillo}
\address{
Departamento de
Matem\'aticas\\
Centro de Investigaci\'on y de Estudios Avanzados del IPN\\
Apartado Postal
14--740 \\
07000 Mexico City, Mexico.
}
\email{djaramillo@math.cinvestav.mx}

\author[M. Vaz Pinto]{Maria Vaz Pinto}
\address{Departamento de Matem\'atica, Instituto Superior T\'ecnico,
Universidade de Lisboa, Avenida Rovisco Pais, 1, 1049-001 Lisboa,
Portugal.
} 
\email{vazpinto@math.tecnico.ulisboa.pt}

\author[R. H. Villarreal]{Rafael H. Villarreal}
\address{
Departamento de
Matem\'aticas\\
Centro de Investigaci\'on y de Estudios
Avanzados del
IPN\\
Apartado Postal
14--740 \\
07000 Mexico City, Mexico
}
\email{vila@math.cinvestav.mx}

\keywords{Evaluation codes, toric codes, minimum distance, projective
torus, footprint, degree, Reed--Muller codes, generalized Hamming
weights, affine variety, finite field, Gr\"obner bases.}
\subjclass[2010]{Primary 13P25; Secondary 14G50, 94B27, 11T71.} 
\begin{abstract} 
The aim of this work is to give degree formulas for the generalized Hamming weights of
evaluation codes and to show lower bounds for these weights.
In particular, we give degree formulas for the generalized Hamming weights 
of Reed--Muller-type codes, and we determine the minimum
distance of toric codes over hypersimplices,
and the 1st and 2nd generalized Hamming weights of squarefree evaluation codes. 
\end{abstract}

\maketitle 

\section{Introduction}\label{intro-section}

Let $S=K[t_1,\ldots,t_s]=\bigoplus_{d=0}^\infty S_d$ be a polynomial ring over a finite
field $K=\mathbb{F}_q$ with the standard grading and let
$X=\{P_1,\ldots,P_m\}$ be a set of distinct points in the affine space
$\mathbb{A}^s:=K^s$. The \textit{evaluation map}, denoted ${\rm ev}$,
is the $K$-linear map given by  
$${\rm ev}\colon S\rightarrow K^{m},\quad
f\mapsto\left(f(P_1),\ldots,f(P_m)\right).
$$ 
\quad The kernel of ${\rm ev}$, denoted $I(X)$, is the 
\textit{vanishing ideal} of $X$ consisting of the polynomials of $S$
that vanish at all points of $X$. This map induces an isomorphism of
$K$-linear spaces between $S/I(X)$ and $K^m$. Let $\mathcal{L}$ be a
linear subspace of $S$ of finite dimension. The image of $\mathcal{L}$
under the evaluation map, denoted $\mathcal{L}_X$, is called an
\textit{evaluation code} on $X$ \cite{stichtenoth,tsfasman}. 

Let $\prec$ be a monomial order on $S$ \cite[p.~54]{CLO} and let $I=I(X)$ be the
vanishing ideal of $X$. The monomials of $S$ are denoted
$t^a:=t_1^{a_1}\cdots t_s^{a_s}$, $a=(a_1,\dots,a_s)$ in $\mathbb{N}^s$, where
$\mathbb{N}=\{0,1,\ldots\}$. We denote the initial monomial of a
non-zero polynomial $f\in S$ by ${\rm in}_\prec(f)$ and the initial
ideal of $I$ by ${\rm in}_\prec(I)$. A monomial $t^a$ is called a 
\textit{standard monomial} of $S/I$, with respect 
to $\prec$, if $t^a\notin{\rm in}_\prec(I)$. 
The \textit{footprint} of $S/I$, denoted
$\Delta_\prec(I)$, is the set of all standard monomials of $S/I$. The
footprint has been used in connection with many kinds of codes 
\cite{geil-2008,geil-hoholdt,geil-pellikaan,Pellikaan}. 

The linear code $\mathcal{L}_X$ is called a \textit{standard
evaluation code} on $X$ relative to $\prec$ if 
$\mathcal{L}$ is a linear subspace 
of $K\Delta_\prec(I)$, the $K$-linear space spanned by $\Delta_\prec(I)$. 
A polynomial $f$ is called
a \textit{standard polynomial} of $S/I$ if $f\neq 0$ and $f$ is in
$K\Delta_\prec(I)$. 
As the field $K$ is finite, there are only a finite number of standard
polynomials. 

The aim of this work is to introduce general methods, similar to those of
\cite{min-dis-generalized,rth-footprint,hilbert-min-dis}, to study the
basic parameters of the family of evaluation codes and
those of certain interesting subfamilies, such as, affine and
projective Reed--Muller-type codes, generalized toric codes, toric codes over
hypersimplices and squarefree evaluation codes. This will also 
allow us to gain insight of the 
geometry of affine varieties and systems of
polynomial equations over finite fields
\cite{Becker-Weispfenning,singular,cocoa-book,Kreuzer-linear-algebra}.
We study standard evaluation  
codes first and then, using a monomial order, 
we show how to transform an evaluation code into
a standard evaluation code. 

If $F$ is a finite subset of $S$, the \textit{affine variety} of $F$ in
$X$, denoted $V_X(F)$, is the set of all
$\alpha\in X$ such that $f(\alpha)=0$ for all $f\in F$.  
The ideal 
$$(I\colon(F)):=\{g\in S\, \vert\, g(F) \in I \}
$$
is referred to as a \textit{colon ideal}. This ideal is useful to determine whether or not
the affine variety $V_X(F)$ is non-empty (Lemma~\ref{vila-delio-feb27-20}). 

The \textit{degree} of the quotient ring $S/I$, denoted $\deg(S/I)$, is
defined using Hilbert functions at the beginning of
Section~\ref{degree-variety-section}.  
This invariant plays a unifying role in the theory of
affine and projective varieties over finite fields. For instance, 
using the degree, one has similar
formulas for $|V_X(F)|$ when $X$ 
is a set of affine or
projective points (cf.
Lemma~\ref{degree-formula-zeros-affine} and
\cite[Lemma~3.4]{rth-footprint}).
The footprint of $S/I$ combined with
the degree (Theorems~\ref{affine-zeros-formula} and
\ref{degree-initial-footprint-affine}) 
will be used to compute and to find lower bounds for the
generalized Hamming weights of evaluation  codes over $X$. 

We now turn our attention to standard evaluation codes and present 
degree formulas, and degree-footprint lower bounds formulas, for their generalized Hamming
weights. The \textit{parameters} of the linear
code $\mathcal{L}_X$ that we consider are:
\begin{itemize}
\item[(a)] \textit{length}: $|X|$,
\item[(b)] \textit{dimension}: $\dim_K(\mathcal{L}_X)$, and 
\item[(c)] $r$-th \textit{generalized Hamming weight}: 
$\delta_r(\mathcal{L}_X)$. 
\end{itemize}

For convenience we recall the notion of 
generalized Hamming weight of a linear code \cite{helleseth,wei}. Let $C$ be a $[m,k]$ {\it linear
code} of {\it length} $m$ and {\it dimension} $k$, 
that is, $C$ is a linear subspace of $K^m$ with $k=\dim_K(C)$. Let $1\leq r\leq k$ be an integer.  
Given a subcode $D$ of $C$ (that is, $D$ is a linear subspace of $C$),
the {\it support\/} $\chi(D)$ of $D$ is the set   
$$
\chi(D):=\{i\,\vert\, \exists\, (a_1,\ldots,a_m)\in D,\, a_i\neq 0\}.
$$
\quad The $r$-th \textit{generalized Hamming weight} of $C$, denoted
$\delta_r(C)$, is the size of the smallest support of an
$r$-dimensional subcode. If $r=1$, $\delta_r(C)$ is the \textit{minimum
distance} of $C$ and is denoted simply by $\delta(C)$. Generalized Hamming weights
have received a lot of attention; see
\cite{carvalho,Huffman-Pless,Johnsen,tsfasman,
wei} and the
references therein. The study of these weights is related to 
trellis coding, $t$--resilient functions, and was motivated by 
some applications from cryptography \cite{wei}. 

There are combinatorial formulas for the generalized Hamming weights
of some interesting families of evaluation codes
\cite{Beelen-RM,carvalho-lopez-lopez,rth-footprint,Pellikaan,weights-matroid}.
The work done 
by Heijnen and Pellikaan 
\cite{Pellikaan} relates footprints and generalized Hamming weights and
introduce methods to study certain evaluation codes
(cf.~\cite[Section~7]{Pellikaan}). These methods were used in  
\cite{GHWCartesian} to determine the generalized Hamming weights of affine
Cartesian codes. 

If $F\subset S$, the $K$-linear subspace 
of $S$ spanned by $F$ is denoted by $KF$. Let $\prec$ be a monomial
order on $S$. Given a linear subspace
$\mathcal{L}$ of $K\Delta_\prec(I)$ and an integer $1\leq
r\leq\dim_K(\mathcal{L})$, let $\mathcal{L}^*$
be the set of nonzero elements of $\mathcal{L}$, 
and let $\mathcal{L}_{\prec,r}$ be
the set of all subsets $F=\{f_1,\ldots,f_r\}$ of $\mathcal{L}^*$ such 
that ${\rm in}_\prec(f_1),\ldots,{\rm in}_\prec(f_r)$ are distinct
monomials and $f_i$ is monic for $i=1,\ldots,r$. 

One of the main result of Section~\ref{GHW-section} is the following
degree formula for  
the $r$-th generalized Hamming weight $\delta_r(\mathcal{L}_X)$ of a
standard evaluation code $\mathcal{L}_X$. 

\noindent {\bf Theorem~\ref{GHW-standard-code}.}\textit{ 
Let $X$ be a subset of $\mathbb{A}^s$, let $I$ be the vanishing
ideal of $X$, let $\mathcal{L}$ be a linear
subspace of $K\Delta_\prec(I)$, and let $\mathcal{L}_X$ be the standard evaluation code
on $X$ relative to $\prec$. Then 
$$
\delta_r(\mathcal{L}_X)=\deg(S/I)-\max\{\deg(S/(I,F))\vert\,
F\in\mathcal{L}_{\prec,r}\}\ \mbox{ for }\ 1\leq r\leq\dim_K(\mathcal{L}_X).
$$
}
\quad This theorem can be applied to any evaluation code
$\mathcal{L}_X$ by constructing a 
standard evaluation code $\widetilde{\mathcal{L}}_X$ on $X$, relative to
a monomial order $\prec$, such that
$\widetilde{\mathcal{L}}_X=\mathcal{L}_X$
(Proposition~\ref{transforming-to-standard-form},
Example~\ref{transforming-example}). We will apply this theorem to 
interesting subfamilies of evaluation codes. 

Fix a degree $d\geq 1$ and let $S_{\leq 
d}=\bigoplus_{i=0}^dS_i$ be the
$K$-linear subspace of $S$ of all polynomials of degree at most $d$. If
$\mathcal{L}$ is equal to $S_{\leq d}$, then the resulting evaluation
code $\mathcal{L}_X$ is called a \textit{Reed-Muller-type code} of degree $d$ on $X$
\cite{duursma-renteria-tapia,GRT} and is denoted by $C_X(d)$. If
$X=\mathbb{A}^s$, we obtain the generalized Reed--Muller code
$\mathcal{R}_q(d,s)$ \cite[p.~524]{Huffman-Pless} or $q$-ary Reed--Muller code
\cite{Pellikaan}. As an
application of Theorem~\ref{GHW-standard-code} we obtain a formula to
compute the $r$-th generalized Hamming weight of $C_X(d)$ using the degree and a
graded monomial order (Corollary~\ref{rth-GHW-affine},
Example~\ref{5points-in-A2}).

The minimum distance of $C_X(d)$ can be
computed using the standard polynomials of $S/I$ of degree at
most $d$ (Corollary~\ref{rth-GHW-affine}). Using this together with 
Proposition~\ref{behavior-hilbert-function},    
we show that the minimum distance
of $C_X(d)$ can be computed recursively using only the standard
polynomials of $S/I$ of degree $d$ (Corollary~\ref{min-dis-affine},
Example~\ref{12points-A3-F3}).

To compute $\delta_r(\mathcal{L}_X)$ is a very difficult problem even
for $r=1$ because $|\mathcal{L}^*|=q^k-1$, where $k$ is the
dimension of $\mathcal{L}_X$. In practice
the formula of Theorem~\ref{GHW-standard-code} can only be used to
compute examples for small values of $r,q,s$, and $k$. Using
Theorem~\ref{GHW-standard-code}, we 
show lower bounds for $\delta_r(\mathcal{L}_X)$, in terms of the degree and
the footprint of $S/I$, which are easier to compute. 

Let $\mathcal{N}_{\prec,r}$ be the 
family of all subsets $N$ of ${\rm in}_\prec(\mathcal{L}^*):=\{{\rm
in}_\prec(f)\vert\, f\in\mathcal{L}^*\}$
with $r$ distinct elements.  The $r$-th \textit{footprint} of the
evaluation code $\mathcal{L}_X$, denoted 
${\rm fp}_r(\mathcal{L}_X)$, 
is given 
by 
$$
{\rm fp}_r(\mathcal{L}_X):=\deg(S/I)-\max\{\deg(S/({\rm
in}_\prec(I),N))\,\vert\,
N\in\mathcal{N}_{\prec,r}\}.
$$
\quad The $r$-th footprint of $\mathcal{L}_X$ is easier to compute
than $\delta_r(\mathcal{L}_X)$ because $|{\rm
in}_\prec(\mathcal{L}^*)|$ is $k$. The other main result of
Section~\ref{GHW-section} is the 
following degree-footprint lower bound for $\delta_r(\mathcal{L}_X)$. 

\noindent {\bf Theorem~\ref{rth-footprint-lower-bound}.}\textit{ 
Let $X$ be a subset of $\mathbb{A}^s$, let $I=I(X)$ be the vanishing
ideal of $X$, let $\mathcal{L}$ be a linear
subspace of $K\Delta_\prec(I)$, and let $\mathcal{L}_X$ be the standard evaluation code
on $X$. Then 
$${\rm fp}_r(\mathcal{L}_X)\leq \delta_r(\mathcal{L}_X)
\ \mbox{ for }\ 1\leq r\leq\dim_K(\mathcal{L}_X).
$$
}
\quad As an application of Theorem~\ref{rth-footprint-lower-bound} we
obtain a lower bound for the
$r$-generalized Hamming weight of $C_X(d)$
(Corollary~\ref{rth-footprint-lower-bound-Reed-Muller},
Example~\ref{affine-torus-A2-F5}).  

The scope of our results include another interesting family of 
evaluation codes that we now introduce. 
Let $\mathbb{X}=\{P_1,\ldots,P_m\}$ be a set of non-zero distinct points in
$(\mathbb{F}_q)^s$ such that the first non-zero
entry of each $P_i$ is $1$. If $\mathcal{L}=S_d$, the evaluation 
code $\mathcal{L}_\mathbb{X}$ 
on $\mathbb{X}$ is called a \textit{projective Reed--Muller-type}
code on $\mathbb{X}$ \cite{GRT}. In particular, by making $\mathcal{L}=S_d$ and
$\mathbb{X}$ equal to the set of all non-zero points of
$(\mathbb{F}_q)^s$ whose first non-zero entry is $1$, we obtain the 
classical projective Reed--Muller code studied by
Lachaud and S{\o}rensen\cite{lachaud,sorensen}. 

As we just saw, one can apply our results on evaluation codes to 
non-classical projective
Reed--Muller-type codes (Example~\ref{10points-P3-F3},
Procedure~\ref{10points-P3-F3-procedure}). 
One can study
Reed-Muller-type codes over a set $X\subset\mathbb{A}^s$ by using 
projective Reed--Muller-type codes over the set
$\mathbb{X}=\{(x,1)\vert\, x\in X\}$,
because the corresponding codes on $X$ and $\mathbb{X}$ have the same
basic parameters \cite{affine-codes}. 

We show examples of how to use Hilbert's Nullstellensatz over 
$\mathbb{F}_q$ \cite{Ghorpade}
(Proposition~\ref{Nullstellensatz-finite}) to estimate the parameters
of evaluation codes over affine varieties 
that are in a broad sense algebraic geometry codes
\cite[p.~192]{tsfasman} (Examples~\ref{Hermitian-curve-example} and 
\ref{elliptic-curve-example}). Then, we give a projective
version over $\mathbb{F}_q$ of Hilbert's Nullstellensatz
(Theorem~\ref{Nullstellensatz-finite-projective}) that can be used in 
the case of evaluation codes over projective varieties defined by a
given set of homogeneous polynomials (Example~\ref{Cox-example}).

The rest of this paper is devoted to show applications 
of the results of Sections~\ref{degree-variety-section} and
\ref{GHW-section} to two other families of evaluation codes that
we now introduce.   

First we introduce toric codes over hypersimplices. 
Let $1\leq d\leq s$ be an integer, $s\geq 2$, 
and let $\mathcal{P}$ be the convex hull in $\mathbb{R}^s$
of all integral points $\mathbf{e}_{i_1}+\cdots+\mathbf{e}_{i_d}$
such that $1\leq i_1<\cdots< i_d\leq s$, 
where $\mathbf{e}_i$ is the $i$-th unit vector in $\mathbb{R}^s$. The lattice polytope
$\mathcal{P}$ is called the $d$-th \textit{hypersimplex} of
$\mathbb{R}^s$ \cite[p.~84]{Stur1}.  
The \textit{affine torus} of the affine space $\mathbb{A}^s:=K^s$ is 
given by $T:=(K^*)^s$, where $K^*$ is the
multiplicative group of the field $K$. The \textit{toric code} of
$\mathcal{P}$ of degree $d$,  
denoted $\mathcal{C}_\mathcal{P}(d)$ or simply
$\mathcal{C}_d$, is the image
of the evaluation map 
\begin{equation}\label{may10-20}
{\rm ev}_d\colon KV_d\rightarrow K^{m},\quad
f\mapsto\left(f(P_1),\ldots,f(P_m)\right),
\end{equation}
where $KV_d$ is the $K$-linear subspace of $S_d$ spanned by
the set $V_d$ of all $t^a$ such that $a\in\mathcal{P}\cap\mathbb{Z}^s$, and
$\{P_1,\ldots,P_m\}$ is the set of all points 
of the affine torus $T$. A monomial $t^a$ of $S$ is in
$KV_d$ if and only if $t^a$ is squarefree and has degree $d$. The set
$V_d$ is precisely the set of squarefree monomials of $S$ of degree $d$. 

Toric codes were introduced by Hansen \cite{Hansen} and have
been actively studied in the last decade, see \cite{Soprunov} and the
references therein. 
These codes are affine-variety
codes in the sense of \cite[p.~1]{Ballico-Elia-Sala}. 
For $q\geq 3$, the toric code $\mathcal{C}_\mathcal{P}(d)$ is a standard evaluation
code on $T$, relative to any monomial order $\prec$, because the
linear space $\mathcal{L}=KV_d$ is spanned by $V_d$, 
all elements of $V_d$ are standard monomials of $S/I(T)$, and
$\mathcal{L}_T=\mathcal{C}_\mathcal{P}(d)$ (Lemma~\ref{toric-squarefree-standard}).  

We solve part of the following problem by determining the minimum distance
of $\mathcal{C}_\mathcal{P}(d)$.

\begin{problem}
Find formulas for the minimum distance or more generally for the
generalized Hamming 
weights of the toric code $\mathcal{C}_\mathcal{P}(d)$.
\end{problem}
We come to our main result on toric codes.

\noindent{\bf Theorem~\ref{Notre-Dame-Cathedral-Apr15-2019}.}\textit{ 
Let $\mathcal{C}_\mathcal{P}(d)$ be the toric code of $\mathcal{P}$
of degree $d$ and let $\delta(\mathcal{C}_\mathcal{P}(d))$ be its minimum distance. Then
$$
\delta(\mathcal{C}_\mathcal{P}(d))=\begin{cases}
(q-2)^d(q-1)^{s-d}&\mbox{ if }d\leq s/2,\, q\geq 3,\\
(q-2)^{s-d}(q-1)^{d}&\mbox{ if }s/2 < d < s, \,q\geq 3,\\
(q-1)^{s} &\mbox{ if } d=s,
\\
1 &\mbox{ if } q=2.
\end{cases}
$$
}
\quad The 2nd generalized Hamming weight of the toric code
$\mathcal{C}_\mathcal{P}(d)$ has 
been recently determined by Patanker and Singh \cite{Patanker-Singh}.

We now introduce the family of squarefree evaluation codes. Elements
of this family are  generalized toric codes in 
the sense of \cite{Little,Ruano,Soprunov}.  Let 
$V_{\leq d}$ be the set of all squarefree
monomials of $S$ of degree at most $d$. If we replace $V_{d}$ by
$V_{\leq d}$ in the evaluation map of Eq.~(\ref{may10-20}), the image of the
resulting map, denoted $\mathcal{C}_{\leq d}$, is called 
a \textit{squarefree evaluation code} of degree $d$ on $T$. If we replace $KV_{d}$ by
$S_{\leq d}$ in the evaluation map of Eq.~(\ref{may10-20}), 
the image of the resulting map is the 
Reed--Muller-type code $C_T(d)$ over the affine torus
$T$. The parameters of $C_T(d)$
have been determined in
\cite{GHWCartesian,rth-footprint,
camps-sarabia-sarmiento-vila,cartesian-codes,ci-codes}.

We solve part of the following problem by 
determining easy to evaluate formulas for the minimum distance and
the 2nd generalized Hamming weight of 
$\mathcal{C}_{\leq d}$.

\begin{problem}
Find formulas for the minimum distance or more generally for the generalized Hamming
weights of the squarefree evaluation code $\mathcal{C}_{\leq d}$.
\end{problem}
The first main result about squarefree evaluation codes is:

\noindent{\bf Theorem~\ref{min-dis-squarefree}.}\textit{ 
If $q\geq 3$, then the minimum distance 
of $\mathcal{C}_{\leq d}$ is $(q-2)^{d}(q-1)^{s-d}$.
}

We obtain formulas to compute the generalized Hamming weights of
squarefree evaluation codes, and also obtain the corresponding
footprint lower bounds
(Corollaries~\ref{rth-GHW-squarefree} and
\ref{rth-footprint-lower-bound-squarefree}). 

The second main result about squarefree evaluation codes is the
following formula.

\noindent{\bf Theorem~\ref{2nd-GHW-squarefree}.}\textit{
If $q\geq 3$, then the second generalized Hamming weight of $\mathcal{C}_{\leq d}$ is
$$
\delta_2(\mathcal{C}_{\leq d})=
\begin{cases}
(q-2)^{s-1}(q-1)&\mbox{ if }\, d=s,\\
(q-2)^d(q-1)^{s-d-1}q&\mbox{ if }\, d<s.
\end{cases}
$$
}
\quad We include one section with some examples illustrating some of our
results (Section~\ref{examples-section}) and an appendix with implementations in
\textit{Macaulay}$2$ \cite{mac2} that show how some
of our results can be used in practice (Appendix~\ref{Appendix}).  

For all unexplained
terminology and additional information we refer to 
\cite{CLO,Sta1,monalg-rev} (for the theory of Gr\"obner bases and Hilbert
functions), and
\cite{Huffman-Pless,MacWilliams-Sloane,tsfasman} (for the theory of
error-correcting codes and linear codes). 

\section{Preliminaries: Affine varieties over finite
sets}\label{degree-variety-section} 
In this section we study affine varieties defined over finite sets. 
The results of this section do not need the hypothesis that the
field is finite.

Let $S=K[t_1,\ldots,t_s]$ be a
polynomial ring over a field $K$ and let $I$ be an ideal of $S$. 
The Krull dimension of $S/I$ is
denoted by $\dim(S/I)$. We say that $I$ has \textit{dimension} $k$ if
$\dim(S/I)$ is equal to $k$. The $K$-linear space of polynomials in $S$ 
(resp. $I$) of degree at most $d$ is denoted by $S_{\leq d}$ (resp.
$I_{\leq d}$). The function
$$
H_I^a(d):=\dim_K(S_{\leq d}/I_{\leq d}),\ \ \ d=0,1,2,\ldots,
$$
is  called the \textit{affine Hilbert function} of $S/I$. 
Let $u=t_{s+1}$ be a new variable and let $I^h\subset S[u]$ 
be the {\it homogenization\/} of $I$, where $S[u]$ is given the 
standard grading. 
One has the following two well-known facts
$$
\dim(S[u]/I^h)=\dim(S/I)+1\mbox{ and } H_I^a(d)=H_{I^h}(d)
\mbox{ for }d\geq 0,
$$
where $H_{I^h}(d)=\dim_K(S[u]/I^h)_d$, see for instance
\cite[Lemma~8.5.4]{monalg-rev}. If $k=\dim(S/I)$, by Hilbert
theorem \cite[p.~58]{Sta1}, there is a unique polynomial
$h^a_I(z)=\sum_{i=0}^{k}a_iz^i\in  
\mathbb{Q}[z]$ of degree $k$ such that
$h^a_I(d)=H_I^a(d)$ for $d\gg 0$. By convention the
degree of the zero polynomial is $-1$.   
The integer $k!\, a_k$, denoted ${\rm deg}(S/I)$, is called the
\textit{degree} of $S/I$.  The degree of $S/I$ is equal to 
$\deg(S[u]/I^{h})$. 
If $k=0$, then $H_I^a(d)=\deg(S/I)=\dim_K(S/I)$ for $d\gg 0$. Note
that the degree of $S/I$ is positive if $I\subsetneq S$ and is $0$
otherwise. 

An element $f\in S$ is called a {\it zero-divisor\/} of $S/I$---as an
$S$-module---if there is
$\overline{0}\neq \overline{a}\in S/I$ such that
$f\overline{a}=\overline{0}$, and $f$ is called {\it regular\/} on
$S/I$ otherwise. Note that $f$ is a zero-divisor of $S/I$ if
and only if $(I\colon f)\neq I$. An associated prime of $I$ is a prime
ideal $\mathfrak{p}$ of $S$ of the form $\mathfrak{p}=(I\colon f)$
for some $f$ in $S$. The radical of $I$ is denoted by ${\rm rad}(I)$.
The ideal $I$ is \textit{radical} if $I={\rm rad}(I)$.

\begin{theorem}{\cite[Lemma~2.1.19,
Corollary~2.1.30]{monalg-rev}}\label{zero-divisors} If $I$ is an
ideal of $S$ and
$I=\mathfrak{q}_1\cap\cdots\cap\mathfrak{q}_m$ is
an irredundant primary decomposition with ${\rm
rad}(\mathfrak{q}_i)=\mathfrak{p}_i$, then the set of zero-divisors
$\mathcal{Z}_S(S/I)$  of $S/I$ is equal to
$\bigcup_{i=1}^m\mathfrak{p}_i$,
and $\mathfrak{p}_1,\ldots,\mathfrak{p}_m$ are the associated primes of
$I$.
\end{theorem}

\begin{proposition}{\rm(\cite{geramita-cayley-bacharach}, 
\cite[p.~411]{affine-codes})}\label{behavior-hilbert-function} Let
$I=I(X)$ be the vanishing ideal of a set $X$ of 
affine points over a field $K$. Then, $H_I^a$ is increasing until
it reaches constant value $|X|$, and $\delta(C_{X}(d))$ is
decreasing, as a function of $d$, until 
it reaches constant value $1$. In particular, $\deg(S/I)=|X|$.
\end{proposition}

The least integer $r_0\geq 0$ such that 
$H_I^a(d)=|X|$ for $d\geq r_0$, denoted ${\rm reg}(H_I^a)$, is the
\textit{index of regularity} of the affine Hilbert function. 

\begin{proposition}{\rm(Additivity of the degree
\cite[Proposition~2.5]{prim-dec-critical})}\label{additivity-of-the-degree}
If $I$ is an ideal of $S$ and 
$I=\mathfrak{q}_1\bigcap\cdots\bigcap\mathfrak{q}_m$ 
is an irredundant primary
decomposition, then\index{degree!is
additive}
$$
\deg(S/I)=\sum_{{\rm ht}(\mathfrak{q}_i)={\rm
ht}(I)}\hspace{-3mm}\deg(S/\mathfrak{q}_i).$$
\end{proposition}

\begin{lemma}{\cite[p.~389]{cocoa-book}}\label{primdec-ixx} 
Let $X$ be a finite subset of
$\mathbb{A}^s$, let $P$ be a point in $X$, $P=(p_1,\ldots,p_s)$, and
let $I_{P}$ be the vanishing ideal 
of $P$. Then, $I_P$ is a prime ideal of height $s$, 
\begin{equation*}
I_P=(t_1-p_1,\ldots,t_s-p_s),\ \deg(S/I_P)=1, 
\end{equation*}
and $I(X)=\bigcap_{P\in X}I_{P}$ is the primary
decomposition of $I(X)$.  
\end{lemma}

\begin{lemma}\label{vila-delio-feb27-20} Let $X$ be a finite subset of 
$\mathbb{A}^s$ over a field $K$ and let $F=\{f_1,\ldots,f_r\}$ be a
set of polynomials of $S$. Then, the following conditions are equivalent.
\begin{enumerate}
\item[(a)] $(I(X)\colon(F))=I(X)$.
\item[(b)] $V_X(F)=\emptyset$.
\item[(c)] $(I(X),F)=S$.
\end{enumerate}
\end{lemma}

\begin{proof} (a) $\Rightarrow$ (b): We can write $X=\{P_1,\ldots,P_m\}$ and
$I(X)=\bigcap_{i=1}^m\mathfrak{p}_i$, where $\mathfrak{p}_i$ is
equal to $I_{P_i}$, the vanishing ideal of $P_i$. We proceed by
contradiction. Assume that $V_X(F)\neq\emptyset$. Pick
$P_k$ in $V_X(F)$. Then, $f_i(P_k)=0$ and $f_i\in\mathfrak{p}_k$ for
all $i$. Note that $(\mathfrak{p}_k\colon(F))=(1)$. Therefore
\[
\bigcap_{i=1}^m\mathfrak{p}_i=I(X)=(I(X)\colon(F))=
\bigcap_{i=1}^m(\mathfrak{p}_i\colon(F))
=\bigcap_{i\neq k}(\mathfrak{p}_i\colon(F))\subset\mathfrak{p}_k.
\]
Hence
$\mathfrak{p}_i\subset(\mathfrak{p}_i\colon(F))\subset\mathfrak{p}_k$ 
for some $i\neq k$, see \cite[p.~74]{monalg-rev}. 
Thus, $\mathfrak{p}_i=\mathfrak{p}_k$, a
contradiction.

(b) $\Rightarrow$ (a): We proceed by contradiction. Assume that
$I(X)\subsetneq(I(X)\colon(F))$. Pick a polynomial
$g$ such that $gf_i\in I(X)$ for all $i$ and $g\notin I(X)$. 
Then, there is $P$ in $X$ such that $g(P)\neq 0$. 
Thus, $f_i(P)=0$ for all $i$, that is, $P\in
V_X(F)$, a contradiction.

(c) $\Rightarrow$ (b): We can write $1=f+\sum_{i=1}^rg_if_i$, where $f\in
I(X)$ and $g_i\in S$ for all $i$. If $V_X(F)\neq\emptyset$, picking
$\alpha\in V_X(F)$ and evaluating the last equality at $\alpha$, we
get $1=0$, a contradiction.

(a) $\Rightarrow$ (c): If $(I(X),F)\subsetneq S$, pick a maximal ideal
$\mathfrak{m}$ of $S$ that contains $(I(X),F)$. Then,
by Lemma~\ref{primdec-ixx} and \cite[2.1.48, p.~74]{monalg-rev},
$\mathfrak{m}=I_{P_k}$ for some $P_k$ in 
$X$. Thus, $F\subset I_{P_k}$ and $P_k\in V_X(F)$, a contradiction
because conditions (a) and (b) are equivalent. 
\end{proof} 

The next result gives a sufficient conditions for an ideal of
dimension zero to be radical. As usual we denote the derivative of a
univariate polynomial $f$ by $f'$.

\begin{lemma}{\rm(Seidenberg's lemma
\cite{seidenberg})}\label{seidenberg-lemma} Let
$I\subsetneq S$ be an ideal of dimension zero. If $I$ contains a
univariate polynomial $f_i\in K[t_i]$ with
$\gcd(f_i,f_i')=1$ for $i=1,\ldots,s$, then $I$ is an intersection of
finitely many maximal ideals, and any proper ideal of $S$ containing
$I$ is a radical ideal.  
\end{lemma}

If $K=\mathbb{F}_q$ is a finite field, by taking $f_j=t_j^q-t_j$ for
$j=1,\ldots,s$, the next result follows directly from Seidenberg's
lemma.

\begin{proposition}\label{feb28-20} Let $X$ be a finite subset of an affine space
$\mathbb{A}^s$ over a field $K$. For each $1\leq j\leq s$, there is a univariate
polynomial $f_j$ in $K[t_j]$ that vanishes at all points of $X$ 
such that $\gcd(f_j,f_j')=1$, and any proper ideal of $S$ containing $I(X)$ 
is a radical ideal. 
\end{proposition}

\begin{proof} Let $P_1,\ldots,P_m$ be the points of $X$. We can write
$P_i=(p_{i,1},\ldots,p_{i,s})$ with $p_{i,j}\in K$ for all $1\leq
i\leq m$ and $1\leq j\leq s$. For each $1\leq j\leq s$ consider the
set 
$$
D_j:=\{p_{i,j}\colon i=1,\ldots,m\}=\{a_{1,j},\ldots,a_{d_j,j}\},
$$
where $a_{1,j},\ldots,a_{d_j,j}$ are distinct elements of $K$ and
$d_j=|D_j|$ for $j=1,\ldots,s$. The univariate polynomials given by 
$$
f_j:=(t_j-a_{1,j})\cdots(t_j-a_{d_j,j}),\ j=1,\ldots,s, 
$$
vanish at all points of $X$. Each $f_j$ is a separable 
polynomial of $K[t_j]$. Hence, $f_j'$ is
relatively prime to $f_j$ \cite[Theorem~4.5, p.~231]{JacI}. Thus, the result follows from
Lemma~\ref{seidenberg-lemma}.
\end{proof}

The next result is an analog of \cite[Lemma~3.4]{rth-footprint} for affine
spaces.

\begin{lemma}{\rm(cf.~\cite[Proposition~6.2.12,
p.~262]{Kreuzer-linear-algebra})}
\label{degree-formula-zeros-affine}
Let $X$ be a finite subset of $\mathbb{A}^s$ and 
let $I(X)$ be its vanishing ideal. If $F$ is a finite subset of $S$,
then
$$
|V_X(F)|=
\begin{cases}
\deg(S/(I(X),F))&\mbox{if }\, (I(X)\colon (F))\neq
I(X),\\ 
0&\mbox{if }\, (I(X)\colon (F))=I(X).
\end{cases}
$$
\end{lemma}

\begin{proof} 
Let $P_1,\ldots,P_m$ be the points of $X$, $m=|X|$, and let $P=P_k$ be a point in
$X$, with $P=(p_1,\ldots,p_s)$. Then, the vanishing ideal $I_P$
of $P$ is a maximal ideal of $S$ of height $s$, 
\begin{equation*}
I_{P}=(t_1-p_1,\ldots,t_s-p_s),\ \deg(S/I_P)=1,
\end{equation*}
and $I(X)=\bigcap_{P\in X}I_P$ is an irredundant primary
decomposition of $I(X)$ (Lemma~\ref{primdec-ixx}). In particular 
the ideal $I(X)$ is an unmixed radical ideal of dimension $0$. 

Assume that $(I(X)\colon (F))\neq I(X)$. If $(I(X)\colon (F))=S$,
then $(F)\subset I(X)$, $X=V_X(F)$, and  
$$
\deg(S/(I(X),F))=\deg(S/I(X))=\sum_{i=1}^m\deg(S/I_{P_i})=|X|=|V_X(F)|.
$$
\quad Hence, we may assume the strict inclusions 
$I(X)\subsetneq(I(X)\colon(F))\subsetneq{S}$. By
Lemma~\ref{vila-delio-feb27-20}, one has $(I(X),F)\subsetneq S$, and by Proposition~\ref{feb28-20},
$(I(X),F)$ is a radical ideal. Any prime ideal $\mathfrak{p}$
containing $(I(X),F)$ is equal to $I_{P_i}$ for some $i$. Therefore,
we may assume that
$$
(I(X),F)=I_{P_1}\textstyle \bigcap\cdots\bigcap I_{P_n},
$$
for some $n\geq 1$, $(F)\subset I_{P_i}$ for $i\leq n$, and $(F)\not\subset
I_{P_i}$ for $i>n$. As a consequence, noticing that $P_i\in V_X(F)$
if and only if $(F)\subset I_{P_i}$ and by additivity of the degree of 
Proposition~\ref{additivity-of-the-degree}, we get
$$ 
|V_X(F)|=\sum_{P_i\in
V_X(F)}\deg(S/I_{P_i})=\sum_{i=1}^n\deg(S/I_{P_i})=\deg(S/(I(X),F)).
$$
\quad Now assume $(I(X)\colon(F))=I(X)$. Then, by
Lemma~\ref{vila-delio-feb27-20}, $V_X(F)=\emptyset$ and
$|V_X(F)|=0$.
\end{proof}

Let $I\subset S$ be an ideal, let $\prec$ be a monomial order, and let
$\Delta_\prec(I)$ be the set of standard monomials of $S/I$. 
The image of $\Delta_\prec(I)$, under the canonical 
map $S\mapsto S/I$, $x\mapsto \overline{x}$, is a basis of $S/I$ as a
$K$-vector space. This is a classical result of Macaulay \cite[Chapter~5]{CLO}. In
particular, $H_I^a(d)$ is the number of standard
monomials of $S/I$ of degree at most $d$. 

\begin{definition}\label{GB-def} 
Let $I$ be an ideal of $S$ and let $\prec$ be a monomial order. 
A subset $\mathcal{G}=\{g_1,\ldots, g_n\}$ of $I$ is called a 
{\it Gr\"obner basis\/} of $I$ if ${\rm
in}_\prec(I)=({\rm in}_\prec(g_1),\ldots,{\rm in}_\prec(g_n))$. 
\end{definition}

\begin{lemma}{\cite[p.~2]{carvalho}}\label{nov6-15} Let $I\subset S$
be an ideal generated by 
$\mathcal{G}=\{g_1,\ldots,g_n\}$, then
$$
\Delta_\prec(I)\subset\Delta_\prec({\rm
in}_\prec(g_1),\ldots,{\rm in}_\prec(g_n)),
$$
with equality if $\mathcal{G}$ is a Gr\"obner basis of $I$.
\end{lemma}

\begin{proof} Take $t^a$ in $\Delta_\prec(I)$. We set $L:=({\rm
in}_\prec(g_1),\ldots,{\rm in}_\prec(g_n))$. Note 
that $\Delta_\prec(L)$ is the set of all monomials that are not
in $L$. If $t^a\notin\Delta_\prec(L)$, then $t^a\in L$, that is, we
can write $t^a=t^c\,{\rm
in}_\prec(g_i)$ for some $i$ 
and some $t^c$. Then, $t^a={\rm
in}_\prec(t^c g_i)$, with $t^c g_i$ in $I$, a contradiction. Thus,
$t^a\in\Delta_\prec(L)$. Assume that $\mathcal{G}$ is a
Gr\"obner basis of $I$, that is, ${\rm in}_\prec(I)=L$. Then,
$\Delta_\prec(I)=\Delta_\prec(L)$.
\end{proof}

\begin{theorem}{\rm(cf.
\cite[Theorem~8.32]{Becker-Weispfenning})}\label{affine-zeros-formula}
Let  $X$ be a finite subset of an affine space $\mathbb{A}^s$ over a
field $K$. 
If $F=\{f_1,\ldots,f_r\}$ is a 
set of polynomials of $S$, then 
$$
|V_X(F)|=\deg(S/(I(X),F))=\dim_K(S/(I(X),F))=|\Delta_\prec(I(X),F)|.
$$
\end{theorem}

\begin{proof} First assume that $(I(X)\colon(F))\neq I(X)$. Then, by
Lemma~\ref{vila-delio-feb27-20}, $(I(X),F)\subsetneq S$. The first
equality follows from 
Lemma~\ref{degree-formula-zeros-affine}, the second
equality follows from the definition of the degree since $S/(I(X),F)$
has Krull dimension $0$, and the third equality is Macaulay's theorem 
that $\Delta_\prec(I(X),F)$ is a basis for $S/(I(X),F)$ as a $K$-vector
space \cite[Chapter~5]{CLO}. 

Now assume that $(I(X)\colon(F))=I(X)$. Then, by
Lemma~\ref{vila-delio-feb27-20},  $(I(X),F)=S$ and $V_X(F)=\emptyset$.
Thus, $\deg(S/(I(X),F))=0$, $\Delta_\prec(I(X),F)=\emptyset$, and all
numbers in the equality above are $0$. 
\end{proof}

For vanishing ideals, the next result is the affine analog of 
\cite[Lemma~4.1]{rth-footprint}.

\begin{theorem}\label{degree-initial-footprint-affine} Let $X$ be a finite
subset of $\mathbb{A}^s$, let $I=I(X)$ be the vanishing ideal of $X$, and let $\prec$ be 
a monomial order. If $F$ is a finite set of polynomials of $S$ and
$(I\colon (F))\neq I$, then 
$$
|V_X(F)|=\deg(S/(I,F))\leq\deg(S/({\rm
in}_\prec(I),{\rm in}_\prec(F)))\leq\deg(S/I)=|X|,
$$
and $\deg(S/(I,F))<\deg(S/I)$ if $(F)\not\subset I$.
\end{theorem}

\begin{proof} The equality on the left follows from
Lemma~\ref{degree-formula-zeros-affine} and the 
equality on the right follows from Lemma~\ref{primdec-ixx} and the
additivity of the degree of 
Proposition~\ref{additivity-of-the-degree}. We set $J=(I,F)$ and
$L=({\rm in}_\prec(I),{\rm 
in}_\prec(F))$, where ${\rm in}_\prec(F)=\{{\rm in}_\prec(f)\vert\, f\in F\}$. As 
$I\subsetneq(I\colon(F))$, we can pick $a$ in $(I\colon(F))$ and $a$
not in $I$. Then, $f \overline{a}=\overline{0}$ for all $f\in(F)$.
Thus, all elements of $(F)$ are zero divisors of 
$S/I$. Hence, as $I$ is a finite intersection of maximal ideals of
$S$, by Theorem~\ref{zero-divisors} and \cite[2.1.49,
p.~74]{monalg-rev}, there is an
associated prime ideal $\mathfrak{p}$ of $S/I$ such that
$(F)\subset\mathfrak{p}$. Thus, $J\subset\mathfrak{p}\subsetneq S$. The rings
$S/J$, $S/L$, and $S/I$ have Krull dimension $0$ since
$\dim(S/I)=\dim(S/{\rm in}_\prec(I))=0$. 
Pick a Gr\"obner basis $\mathcal{G}=\{g_1,\ldots,g_n\}$ of $I$. Then, 
$J$ is generated by $\mathcal{G}\cup F$ and by Lemma~\ref{nov6-15}
one has the inclusions
\begin{eqnarray*}
& &\Delta_\prec(J)=\Delta_\prec(I,F)\subset\Delta_\prec({\rm in}_\prec(g_1),\ldots,{\rm
in}_\prec(g_n),{\rm in}_\prec(F))=\\
& &\ \ \ \ \  \Delta_\prec({\rm in}_\prec(I),{\rm
in}_\prec(F))=\Delta_\prec(L)\subset \Delta_\prec({\rm in}_\prec(g_1),\ldots,{\rm
in}_\prec(g_n))=\Delta_\prec(I).
\end{eqnarray*}
\quad Thus, $\Delta_\prec(J)\subset \Delta_\prec(L)\subset \Delta_\prec(I)$.
Recall that $H_I^a(d)$, the affine Hilbert function of $I$ at $d$, is the number of standard
monomials of degree at most $d$. Hence, $H_J^a(d)\leq H_L^a(d)\leq
H_I^a(d)$ for
$d\geq 0$. Then, by Hilbert theorem \cite[p.~58]{Sta1}, $H_J^a$,
$H_L^a$, $H_I^a$ are polynomial functions of degree 
equal to $\dim(S/I)=0$, and so they become constant for $d\gg 0$. Thus,   
$$
H_J^a(d)=\deg(S/J)\leq H_L^a(d)=\deg(S/L)\leq
H_I^a(d)=\deg(S/I)
$$
for $d\gg 0$, that is, $\deg(S/J)\leq\deg(S/L)\leq\deg(S/I)$. 
Now, assume that $(F)\not\subset I$. As $I$ is a radical ideal, there is at
least one minimal prime of $I$ that does not contain $(F)$.  By
Proposition~\ref{feb28-20}, 
$J$ is a radical ideal. Hence, ${\rm Ass}(J)\subsetneq{\rm Ass}(I)$,
and using the additivity of the degree of 
Proposition~\ref{additivity-of-the-degree}, we get
$$
\deg(S/J)=\sum_{{\mathfrak{p}\in\rm Ass}(J)}\deg(S/\mathfrak{p})
<\sum_{{\mathfrak{p}\in\rm Ass}(I)}\deg(S/\mathfrak{p})=\deg(S/I),
$$
and consequently $\deg(S/J)<\deg(S/I)$.
\end{proof}

\section{Generalized Hamming weights of evaluation
codes}\label{GHW-section}
In this section we give formulas, in terms of the degree and a
graded monomial order, for the generalized Hamming weights of standard
evaluation codes, and show degree-footprint lower bounds for these
weights which are much easier to compute. To avoid
repetitions, we continue to employ 
the notations and definitions used in Sections~\ref{intro-section} 
and \ref{degree-variety-section}.  Throughout this section
we assume that $K$ is a finite field $\mathbb{F}_q$. 

\begin{lemma}\label{apr26-20} 
Let $\mathcal{L}_X$ be a standard evaluation code on $X$
relative to a monomial order $\prec$. Then, $\mathcal{L}\cap I(X)=(0)$ 
and $\mathcal{L}\simeq\mathcal{L}_X$.
\end{lemma}

\begin{proof} We set $I=I(X)$. Take $f\in\mathcal{L}\cap I$. If $f\neq 0$, 
then ${\rm in}_\prec(f)\in{\rm in}_\prec(I)$, a contradiction since
all monomials of $f$ are standard monomial of $S/I$. Hence, the
evaluation maps gives an isomorphism between $\mathcal{L}$ and
$\mathcal{L}_X$. 
\end{proof}

Let $I$ be an ideal of $S$, let $\prec$ be a monomial order on $S$, let
$\mathcal{L}$ be a linear subspace of $K\Delta_\prec(I)$, let
$\mathcal{L}_r$ be the set of all linearly independent subsets 
$F=\{f_1,\ldots,f_r\}$ of $\mathcal{L}$ with $r$ elements, 
let $\mathcal{L}_{\prec,r}$ be
the set of all subsets $F=\{f_1,\ldots,f_r\}$ of
$\mathcal{L}^*=\mathcal{L}\setminus\{0\}$ such 
that ${\rm in}_\prec(f_1),\ldots,{\rm in}_\prec(f_r)$ are distinct
monomials and $f_i$ is monic for $i=1,\ldots,r$. The
set of all standard polynomials of $S/I$ is denoted by
$\Delta_\prec^p(I)$, that is,
$\Delta_\prec^p(I)$ is equal to $(K\Delta_\prec(I))\setminus\{0\}$.

\begin{lemma}\label{distinct-implies-li}
Let $\mathcal{L}$ be a $K$-linear subspace of $S$ spanned by a finite subset of
$\Delta_\prec^p(I)$, and let
$F=\{f_1,\ldots,f_r\}$ be a subset of
$\mathcal{L}^*=\mathcal{L}\setminus\{0\}$. The following hold. 
\begin{enumerate}
\item[\rm(a)] If $f_1,\ldots,f_r$ are linearly independent over
$K$, then there is a set $G=\{g_1,\ldots,g_r\}\subset\mathcal{L}^*$
such that $KF=KG$, ${\rm in}_\prec(g_1),\ldots,{\rm 
in}_\prec(g_r)$ distinct, and ${\rm
in}_\prec(f_i)\succeq{\rm in}_\prec(g_i)$ for all $i$.
\item[\rm(b)] If ${\rm in}_\prec(f_1),\ldots,{\rm in}_\prec(f_r)$
are distinct, then $f_1,\ldots,f_r$ are
linearly independent over $K$.
\item[\rm(c)] $\mathcal{L}_{\prec,r}\subset\mathcal{L}_{r}$, and if
$F$ is in $\mathcal{L}_r$, then there is   
$G$ in $\mathcal{L}_{\prec,r}$ such that $KF=KG$.
\end{enumerate}
\end{lemma}

\begin{proof} 
(a): Note that all elements of $\mathcal{L}^*$ are standard
polynomials of $S/I$. We proceed by induction on $r$. The case $r=1$ is clear. Assume that $r>1$.
Permuting the $f_i$'s if necessary we may assume that ${\rm
in}_\prec(f_1)\succeq\cdots\succeq{\rm
in}_\prec(f_r)$. 

Case ($\mathrm{a}_1$): Assume that ${\rm in}_\prec(f_1)\succ{\rm
in}_\prec(f_2)$. By applying the induction hypothesis to
the set $F'=\{f_2,\ldots,f_r\}$, we obtain a set
$G'=\{g_2,\ldots,g_r\}\subset\mathcal{L}^*$ such that $KF'$ is equal
to $KG'$, the monomials ${\rm in}_\prec(g_2),\ldots,{\rm
in}_\prec(g_r)$ are distinct, and ${\rm
in}_\prec(f_i)\succeq{\rm in}_\prec(g_i)$ for $i=2,\ldots,r$. Setting
$g_1=f_1$ and $G=G'\cup\{g_1\}$, we get $KF=KG$, and the monomial
${\rm in}_\prec(g_1)$ is distinct from ${\rm
in}_\prec(g_2),\ldots,{\rm in}_\prec(g_r)$ because ${\rm
in}_\prec(f_1)\succ {\rm in}_\prec(f_i)\succeq {\rm in}_\prec(g_i)$
for $i=2,\ldots,r$.

Case ($\mathrm{a}_2$): Assume there is $k\geq 2$ such
that ${\rm in}_\prec(f_1)={\rm in}_\prec(f_i)$ for $i\leq k$ and 
${\rm in}_\prec(f_1)\succ {\rm in}_\prec(f_i)$ for $i>k$. We set
$h_i=f_1-f_i$ for $i=2,\ldots,k$ and $h_i=f_i$ for $i=k+1,\ldots,r$.
Note that ${\rm in}_\prec(f_1)\succ {\rm in}_\prec(h_i)$ for $i\geq
2$, $h_2,\ldots,h_r$ are in $\mathcal{L}^*$, and
$h_2,\ldots,h_r$ are linearly independent over
$K$. By applying the induction hypothesis to
the set $H=\{h_2,\ldots,h_r\}$, we obtain a set
$G'=\{g_2,\ldots,g_r\}\subset\mathcal{L}^*$ such
that $KH=KG'$, ${\rm in}_\prec(g_2),\ldots,{\rm
in}_\prec(g_r)$ distinct, and ${\rm
in}_\prec(h_i)\succeq{\rm in}_\prec(g_i)$ for $i=2,\ldots,r$. Setting
$g_1=f_1$ and $G=G'\cup\{g_1\}$, we get $KF=KG$, and the monomial
${\rm in}_\prec(g_1)$ is distinct from ${\rm
in}_\prec(g_2),\ldots,{\rm in}_\prec(g_r)$ because 
${\rm in}_\prec(f_1)\succ {\rm in}_\prec(h_i)\succeq{\rm
in}_\prec(g_i)$ 
for $i\geq 2$.

(b): By hypothesis, ${\rm in}_\prec(f_1)\succ\cdots\succ{\rm
in}_\prec(f_r)$. 
Assume that
$\sum_{i=1}^r\lambda_if_i=0$, $\lambda_i\in K$
for all $i$. We proceed by
contradiction assuming $\lambda_1=\cdots=\lambda_{k-1}=0$ and
$\lambda_k\neq 0$ for some $k$. Setting $f:=\sum_{i=k}^r\lambda_if_i$,
we get ${\rm in}_\prec(f)={\rm
in}_\prec(f_k)$ and $f\neq 0$, a contradiction.

(c): This follows at once from parts (a) and (b).
\end{proof}

Let $C$ be a $[m,k]$ \textit{linear code} of {\it length} $m$ and
{\it dimension} $k$ over a finite field $K=\mathbb{F}_q$, and let
$1\leq r\leq k$ be an integer. Given a subcode $D$ of $C$, the {\it
support\/} $\chi(D)$ of $D$ is the set    
$$
\chi(D):=\{i\,\vert\, \exists\, (a_1,\ldots,a_m)\in D,\, a_i\neq 0\}.
$$
\quad The {\it support\/} $\chi(\beta)$ of a vector $\beta\in K^m$ 
is $\chi(K\beta)$, that is, $\chi(\beta)$ is the set of all non-zero
entries of $\beta$. The $r$-th {\it generalized Hamming weight\/} of $C$, denoted 
$\delta_r(C)$, is the size of the smallest support of an
$r$-dimensional subcode: 
$$
\delta_r(C):=\min\{|\chi(D)|\,\colon\, D\mbox{ is a subcode of
}C\mbox{ with }\dim_K(D)=r\}. 
$$
\quad The {\it weight hierarchy\/} of $C$ is the sequence
$(\delta_1(C),\ldots,\delta_k(C))$. The integer $\delta_1(C)$ is the minimum
distance of $C$ and is denoted by $\delta(C)$. 
 According to \cite[Theorem~1,
Corollary~1]{wei}   
the weight hierarchy is an
increasing sequence 
$$
1\leq\delta_1(C)<\cdots<\delta_k(C)\leq m,
$$
and $\delta_r(C)\leq m-k+r$ for $r=1,\ldots,k$. For $r=1$ this is the
Singleton bound for the minimum distance. Notice that 
$\delta_r(C)\geq r$. 

\begin{lemma}{\cite[Lemma~2.1]{rth-footprint}}\label{seminar} Let $D$
be a subcode of $C$ of dimension $r\geq 1$. If  
$\beta_1,\ldots,\beta_r$ is a $K$-basis for $D$ with
$\beta_i=(\beta_{i,1},\ldots,\beta_{i,m})$ for $i=1,\ldots,r$, then 
$\chi(D)=\bigcup_{i=1}^r\chi(\beta_i)$ and the number of elements of
$\chi(D)$ is the number of non-zero columns of the matrix:
$$   
\left[\begin{matrix}
\beta_{1,1}&\cdots&\beta_{1,i}&\cdots&\beta_{1,m}\\
\beta_{2,1}&\cdots&\beta_{2,i}&\cdots&\beta_{2,m}\\
\vdots&\cdots&\vdots&\cdots&\vdots\\
\beta_{r,1}&\cdots&\beta_{r,i}&\cdots&\beta_{r,m}
\end{matrix}\right].
$$
\end{lemma}

One of the main result of this section is the following degree formula for 
the $r$-th generalized Hamming weight $\delta_r(\mathcal{L}_X)$ of a
standard evaluation code $\mathcal{L}_X$. 

\begin{theorem}\label{GHW-standard-code}
Let $X$ be a subset of $\mathbb{A}^s$, let $I$ be the vanishing
ideal of $X$, let $\mathcal{L}$ be a linear
subspace of $K\Delta_\prec(I)$, and let $\mathcal{L}_X$ be the standard evaluation code
on $X$ relative to $\prec$. Then 
$$
\delta_r(\mathcal{L}_X)=\deg(S/I)-\max\{\deg(S/(I,F))\vert\,
F\in\mathcal{L}_{\prec,r}\}\ \mbox{ for }\ 1\leq r\leq\dim_K(\mathcal{L}_X).
$$
\end{theorem}

\demo 
Let $P_1,\ldots,P_m$ be the points of $X$ and let $D$ be a subcode of $\mathcal{L}_X$ of 
dimension $r$. The evaluation map
${\rm ev}$ induces an isomorphism of $K$-vector spaces between $\mathcal{L}$ and 
$\mathcal{L}_X$ (Lemma~\ref{apr26-20}). Hence, by Lemma~\ref{seminar}, there are 
$f_1,\ldots,f_r$ linearly independent elements
of $\mathcal{L}$, that is, the set $F:=\{f_1,\ldots,f_r\}$ is in
$\mathcal{L}_{r}$, such that $D=\bigoplus_{i=1}^rK\beta_i$, where 
$\beta_i$ is $(f_i(P_1),\ldots,f_i(P_m))$, and the support 
$\chi(D)$ of $D$ is equal to $\bigcup_{i=1}^r\chi(\beta_i)$. 
Consider the matrix $A$ with rows $\beta_1, \ldots,\beta_r$. The
$j$-th column of $A$ is not zero  
if and only if $P_j$ is in $X\setminus
V_X(F)$. Therefore, since the number of non-zero 
columns of $A$ is $|\chi(D)|$ (Lemma~\ref{seminar}), we get:
\begin{equation}\label{mar19-20}
|\chi(D)|=|X\setminus
V_X(F)|.
\end{equation}
\quad Conversely let $F=\{f_1,\ldots,f_r\}$ be a set in
$\mathcal{L}_r$, then there is a subcode $D$ of
$\mathcal{L}_X$ of dimension $r$ with $|\chi(D)|=|X\setminus V_X(F)|$. Indeed, setting
$$\beta_i:=(f_i(P_1),\ldots,f_i(P_m))\ \mbox{ for }\
i=1,\ldots,r\mbox{ and } D:=K\beta_1+\cdots+K\beta_r,
$$
and using Lemma~\ref{seminar}, we obtain that
$|\chi(D)|=|X\setminus V_X(F)|$.  
The affine varieties defined by the elements of $\mathcal{L}_{r}$ and
$\mathcal{L}_{\prec,r}$ are the same: 
\begin{equation}\label{mar19-20-2} 
\{V_X(F)\colon F\in\mathcal{L}_{r}\}=\{V_X(F)\colon
F\in\mathcal{L}_{\prec,r}\}.
\end{equation}
\quad Indeed, the inclusion ``$\supset$'' is clear since
$\mathcal{L}_{\prec,r}\subset \mathcal{L}_{r}$
(Lemma~\ref{distinct-implies-li}). To show the 
inclusion ``$\subset$'' take $F\in \mathcal{L}_{r}$. Then, by
Lemma~\ref{distinct-implies-li}, there is $G\in \mathcal{L}_{\prec,r}$ such
that $KF=KG$. Thus, $V_X(F)=V_X(G)$ with $G\in
\mathcal{L}_{\prec,r}$, that is, $V_X(F)$ is in the right hand
side of Eq.~(\ref{mar19-20-2}). By
Proposition~\ref{behavior-hilbert-function}, $|X|=\deg(S/I)$. Hence,
by Eqs.~(\ref{mar19-20})--(\ref{mar19-20-2}) and  
Lemma~\ref{degree-formula-zeros-affine}, we obtain
\begin{eqnarray*}
\delta_r(\mathcal{L}_X)&=&\min\{|\chi(D)|\,\colon\, D\mbox{ is a subcode of
}\mathcal{L}_X \mbox{ and } \dim_K(D)=r \}
\\
&=&\min\{|X\setminus
V_X(F)|:\, F\in\mathcal{L}_{r}\}\\
&=&|X|-\max\{|V_X(F)|\colon\, F\in\mathcal{L}_{r}\}\\
&=&
\deg(S/I)-\max\{|V_X(F)|\colon\, F\in\mathcal{L}_{\prec,r}\}\\
&=&\deg(S/I)-\max\{\deg(S/(I,F)):\,
F\in\mathcal{L}_{\prec,r}\}.\quad\Box
\end{eqnarray*}

The next proposition allows us to transform any 
evaluation code into a standard
evaluation code relative to a monomial order, that is, one can 
apply Theorem~\ref{GHW-standard-code} to any evaluation code after
picking a monomial order and making a suitable transformation. 
We illustrate the case of generalized
toric codes in Example~\ref{transforming-example}
(Procedure~\ref{transforming-procedure}) and the case 
of projective
Reed--Muller-type codes in Example~\ref{10points-P3-F3}
(Procedure~\ref{10points-P3-F3-procedure}).

\begin{proposition}\label{transforming-to-standard-form} 
Let $\mathcal{L}_X$ be an evaluation code on $X$ and let $\prec$ be a
monomial order. Then, there is a standard evaluation code
$\widetilde{\mathcal{L}}_X$ on $X$ relative to $\prec$ such that 
$\widetilde{\mathcal{L}}_X=\mathcal{L}_X$.
\end{proposition}

\begin{proof} Let $X$ be the set of points
$\{P_1,\ldots,P_m\}$ of the affine space 
$\mathbb{A}^s$and let $\mathcal{G}$ be a Gr\"obner basis of
$I=I(X)$. Pick a $K$-basis $\{f_1,\ldots,f_k\}$ of $\mathcal{L}$.
For each $i$, let $r_i$ be the remainder on division of $f_i$ by
$\mathcal{G}$, that is, by the division algorithm \cite[Theorem~3,
p.~63]{CLO}, for each $i$ we can write $f_i=h_i+r_i$, where $h_i\in
I$, and $r_i=0$ or $r_i$ is a standard polynomial of $S/I$. We set
$$
\widetilde{\mathcal{L}}:=Kr_1+\cdots+Kr_k.
$$
\quad The evaluation code $\widetilde{\mathcal{L}}_X$ is a standard
evaluation code on $X$ relative to $\prec$ since
$\widetilde{\mathcal{L}}$ is a linear subspace of $K\Delta_\prec(I)$.
To show the inclusion
${\mathcal{L}}_X\subset\widetilde{\mathcal{L}}_X$ take a point $P$ in
${\mathcal{L}}_X$. Then, $P$ is equal to $(f(P_1),\ldots,f(P_m))$ for some
$f\in\mathcal{L}$. Using the equations $f_i=h_i+r_i$, $i=1,\ldots,k$,
we can write $f=h+r$, where $h\in I$ and
$r\in\widetilde{\mathcal{L}}$. Hence, $P$ is equal to
$(r(P_1),\ldots,r(P_m))$, that is, $P\in \widetilde{\mathcal{L}}_X$. 
To show the inclusion
$\widetilde{\mathcal{L}}_X\subset{\mathcal{L}}_X$ take a point $Q$ in 
$\widetilde{\mathcal{L}}_X$. Then, $Q$ is equal to $(g(P_1),\ldots,g(P_m))$ for some
$g\in\widetilde{\mathcal{L}}$. Using the equations $f_i=h_i+r_i$, $i=1,\ldots,k$,
we can write $g=g_1+g_2$, where $g_1\in I$ and
$g_2\in{\mathcal{L}}$. Hence, $Q$ is equal to
$(g_2(P_1),\ldots,g_2(P_m))$, that is, $Q\in{\mathcal{L}}_X$. 
\end{proof}

Let $\mathcal{L}_X$ be a standard evaluation code on $X$ relative to
$\prec$ of dimension $k$. Fix an integer $1\leq r\leq k$. The
generalized Hamming weights of $\mathcal{L}_X$ are hard to compute
because the size of $\mathcal{L}_{\prec,r}$ could be very large as we now
explain.  The 
\textit{Grassmannian} of $\mathcal{L}$, denoted $G(r,k)$, is the set
of $r$-dimensional subspaces of $\mathcal{L}$. Consider the
equivalence relation on $\mathcal{L}_{\prec,r}$ given by: 
$F\sim G$ if and only if $KF=KG$. The map 
$$
\vartheta\colon\mathcal{L}_{\prec,r}/\sim\ \ \longrightarrow G(r,k),\ \ \
[F]\longmapsto KF, 
$$
is bijective. This follows from Lemma~\ref{distinct-implies-li}. If
$r=1$, then $|\mathcal{L}_{\prec,r}|=|G(1,k)|=(q^k-1)/(q-1)$. The
following formula
for $G(r,k)$ can be found in \cite[Proposition~1.7.2,
p.~57]{Stanley-2011}:
$$
|G(r,k)|=\frac{(q^k-1)(q^k-q)\cdots(q^k-q^{r-1})}
{(q^r-1)(q^r-q)\cdots(q^r-q^{r-1})}.
$$
\quad Let $F$ and $G$ be in $\mathcal{L}_{\prec,r}$.
If $KF=KG$, then $\{{\rm in}_\prec(f)\vert\, f\in F\}$ is equal to 
$\{{\rm in}_\prec(g)\vert\, g\in G\}$. The converse does not hold in
general but it may hold for some special families of standard
evaluation codes.

One of the applications of
Theorem~\ref{GHW-standard-code} is a formula for   
the $r$-th generalized Hamming weight $\delta_r(C_X(d))$ of an affine
Reed--Muller-type code $C_X(d)$ that can be used to compute
$\delta_r(C_X(d))$ for small values of $q,r,d,s$  using  
the software system {\it Macaulay\/}$2$ \cite{mac2}
(Example~\ref{5points-in-A2}). In the next result we assume that
$\prec$ is a graded monomial order, that is,
monomials are first compared by their total degrees
(Example~\ref{lex-example}). For use below, let $\mathcal{F}_{\prec,d,r}$ be the set of  
all $F=\{f_1,\ldots,f_r\}$ contained in $S_{\leq d}$ such 
that $f_i$ is a standard monic polynomial of $S/I$ for all $i$ and 
${\rm in}_\prec(f_1),\ldots,{\rm in}_\prec(f_r)$ are distinct
monomials.

\begin{corollary}\label{rth-GHW-affine}
Let $K=\mathbb{F}_q$ be a finite field, let $X$ be a subset
of $\mathbb{A}^s$, and let $I=I(X)$ be the vanishing ideal of $X$. If
$\prec$ is a graded monomial order, then  
$$
\delta_r(C_X(d))=\deg(S/I)-\max\{\deg(S/(I,F))\vert\,
F\in\mathcal{F}_{\prec, d,r}\}\ \mbox{ for }\ 1\leq r\leq H_I^a(d).
$$
\end{corollary}
\begin{proof} 
Let $P_1,\ldots,P_m$ be the points of $X$ and 
let $\mathcal{L}=K\Delta_\prec(I)_{\leq d}$ be the linear space 
of standard polynomials of $S/I$ of degree at most $d$ together with
the zero vector. Note that
$\mathcal{L}_{\prec,r}=\mathcal{F}_{\prec,d,r}$. Hence, by
Theorem~\ref{GHW-standard-code}, we need only show that
$\mathcal{L}_X=C_X(d)$. Clearly, $\mathcal{L}_X\subset C_X(d)$. To
show the other inclusion take a point $Q$ in $C_X(d)$, that is, 
$Q=(f(P_1),\ldots,f(P_m))$  for
some $f\in S_{\leq d}$. As
$\prec$ is a graded monomial order, by  
the division algorithm \cite[Theorem~3, p.~63]{CLO}, $f$ can be written as $f=p+h$,
where $p$ is in $I_{\leq d}$ and $h$ is a $K$-linear combination of standard
monomials of degree at most $d$. Hence, $Q=(h(P_1),\ldots,h(P_m))$ 
and $Q$ is in $\mathcal{L}_X$. 
\end{proof}

The minimum distance of the linear code $C_X(d)$ is $\delta_1(C_X(d))$ and can be
computed using the standard monic polynomials of $S/I$ of degree at
most $d$ (Corollary~\ref{rth-GHW-affine}). The next result can be
used to compute the minimum distance 
of $C_X(d)$ recursively using only the standard
monic polynomials of $S/I$ of degree $d$ (Example~\ref{12points-A3-F3}). 

\begin{corollary}\label{min-dis-affine}
Let $K=\mathbb{F}_q$ be a finite field, let $X$ be a subset
of $\mathbb{A}^s$, let $I=I(X)$ be the vanishing ideal of $X$, let
$\prec$ be a graded monomial order, and let $\Delta_\prec^p(I)_d$ be
the set of standard monic polynomials of $S/I$ of degree $d$. If 
$d\geq 2$ and 
$\delta(C_X(d-1))>1$, then   
$$
\delta(C_X(d))=\deg(S/I)-\max\{\deg(S/(I,f))\vert\,
f\in\Delta_\prec^p(I)_d\}.
$$
\end{corollary}

\begin{proof} Let $\Delta_\prec^p(I)_{\leq d}$ be the set of standard
monic polynomials of $S/I$ of degree at most $d$. As
$\mathcal{F}_{\prec,d,1}$ is equal 
to $\Delta_\prec^p(I)_{\leq d}$, by Corollary~\ref{rth-GHW-affine}, one
has
\begin{eqnarray}
& &\delta(C_X(d))=\deg(S/I)-\max\{\deg(S/(I,f))\vert\,
f\in\Delta_\prec^p(I)_{\leq d}\}\\ & & \ \  \ \quad\quad\quad
\quad\quad\quad\leq
\deg(S/I)-\max\{\deg(S/(I,f))\vert\, 
f\in\Delta_\prec^p(I)_d\}.\label{apr20-20}
\end{eqnarray}
\quad This proves the inequality ``$\leq$''. To show the inequality
``$\geq$''pick $f$ in $\Delta_\prec^p(I)_{\leq d}$ such that 
$\delta(C_X(d))=\deg(S/I)-\deg(S/(I,f))$. It suffices to show that
$f$ has degree $d$, because then Eq.~(\ref{apr20-20}) becomes an
equality. If $\deg(f)<d$, then by
Proposition~\ref{behavior-hilbert-function}, we get 
$$\delta(C_X(d))<\delta(C_X(d-1))\leq\deg(S/I)-\deg(S/(I,f))=\delta(C_X(d)),$$
a contradiction.
\end{proof}

For non-graded orders we obtain the following upper bound for
$\delta_r(C_X(d))$. 
\begin{corollary}\label{rth-GHW-affine-bound}
Let $K=\mathbb{F}_q$ be a finite field, let $X$ be a subset
of $\mathbb{A}^s$, and let $I=I(X)$ be the vanishing ideal of $X$. If
$\prec$ is a monomial order, then  
$$
\delta_r(C_X(d))\leq \deg(S/I)-\max\{\deg(S/(I,F))\vert\,
F\in\mathcal{F}_{\prec, d,r}\}.
$$
\end{corollary}

\begin{proof} Let $\mathcal{L}=K\Delta_\prec(I)_{\leq d}$ be the linear space 
of standard polynomials of $S/I$ of degree at most $d$ together with
the zero vector. Note that
$\mathcal{L}_{\prec,r}=\mathcal{F}_{\prec,d,r}$. As 
$\mathcal{L}_X\subset C_X(d)$, one has
$\delta_r(C_X(d))\leq\delta_r(\mathcal{L}_X)$. Therefore, the inequality follows from
Theorem~\ref{GHW-standard-code}.
\end{proof}

Fix a monomial order $\prec$ on $S$, let $X$ be a subset of
$\mathbb{A}^s$, and let $I=I(X)$ be its vanishing ideal. 
Given a $K$-linear subspace $\mathcal{L}$ of 
$K\Delta_\prec(I)$, let $\mathcal{N}_{\prec,r}$ be the 
family of all subsets $N$ of ${\rm in}_\prec(\mathcal{L}^*):=\{{\rm
in}_\prec(f)\vert\, f\in\mathcal{L}^*\}$
with $r$ distinct elements. 
The $r$-th \textit{footprint} of the standard evaluation code $\mathcal{L}_X$, denoted
${\rm fp}_r(\mathcal{L}_X)$, 
is given 
by 
$$
{\rm fp}_r(\mathcal{L}_X):=\deg(S/I)-\max\{\deg(S/({\rm
in}_\prec(I),N))\,\vert\,
N\in\mathcal{N}_{\prec,r}\}.
$$
\quad The other main result of Section~\ref{GHW-section} is the
following lower bound for $\delta_r(\mathcal{L}_X)$.
\begin{theorem}\label{rth-footprint-lower-bound} 
Let $X$ be a subset of $\mathbb{A}^s$, let $I=I(X)$ be the vanishing
ideal of $X$, let $\mathcal{L}$ be a linear
subspace of $K\Delta_\prec(I)$, and let $\mathcal{L}_X$ be the standard evaluation code
on $X$. Then 
$${\rm fp}_r(\mathcal{L}_X)\leq \delta_r(\mathcal{L}_X)\ \mbox{ for
}d\geq 1\mbox{ and }1\leq r\leq\dim_K(\mathcal{L}_X).
$$
\end{theorem}

\begin{proof} By Theorem~\ref{GHW-standard-code}, there is 
$F\in\mathcal{L}_{\prec,r}$ such that
$\delta_r(\mathcal{L}_X)=\deg(S/I)-\deg(S/(I,F))$. 
Hence, by Theorem~\ref{degree-initial-footprint-affine}, and noticing
that ${\rm in}_\prec(F)\in\mathcal{N}_{\prec,r}$ because
$F\in\mathcal{L}_{\prec,r}$, we 
get 
$$
\deg(S/(I,F))\leq\deg(S/({\rm in}_\prec(I),{\rm in}_\prec(I)))\leq
\max\{\deg(S/({\rm 
in}_\prec(I),N))\,\vert\,
N\in\mathcal{N}_{\prec,r}\}.
$$
\quad Thus, ${\rm fp}_r(\mathcal{L}_X)\leq \delta_r(\mathcal{L}_X)$. 
\end{proof}

\quad Given integers $d,r\geq 1$, let
$\mathcal{M}_{\prec, d, r}$ be the 
set of all subsets $M$ of $\Delta_\prec(I)\cap
S_{\leq d}$
with $r$ distinct elements. 
The $r$-th \textit{footprint} of the Reed--Muller-type code $C_X(d)$,
denoted ${\rm fp}_I(d,r)$, is given 
by 
$$
{\rm fp}_I(d,r):=\deg(S/I)-\max\{\deg(S/({\rm
in}_\prec(I),M))\,\vert\,
M\in\mathcal{M}_{\prec, d,r}\}. 
$$
\quad We come to one of the main applications of
Theorem~\ref{rth-footprint-lower-bound}.

\begin{corollary}\label{rth-footprint-lower-bound-Reed-Muller} 
Let $K=\mathbb{F}_q$ be a finite field, let $X$ be a subset
of $\mathbb{A}^s$, let $I=I(X)$ be the vanishing ideal of $X$, and
let $\prec$ be a graded monomial order. 
Then
$${\rm fp}_{I}(d,r)\leq \delta_r(C_X(d))\ \mbox{ for
}d\geq 1\mbox{ and }1\leq r\leq H_{I}^a(d).$$
\end{corollary}

\begin{proof} Let $\mathcal{L}=K\Delta_\prec(I)_{\leq d}$ be the
$K$-vector space generated by the set 
$\Delta_{\prec}(I)_{\leq d}$ of all standard monomials of $S/I$ of degree at
most $d$. As $\prec$ is a graded monomial order, by the division
algorithm \cite[Theorem~3, p.~63]{CLO}, one has: 
$$
S_{\leq d}=(I\cap S_{\leq d})+K\Delta_{\prec}(I)_{\leq d}\ \mbox{ and }\
\mathcal{L}_X={\rm ev}(K\Delta_{\prec}(I)_{\leq d})=
{\rm ev}(S_{\leq d})=C_X(d),
$$
that is, $C_X(d)$ is the standard evaluation code $\mathcal{L}_X$ on
$X$. Hence, the inequality follows directly from
Theorem~\ref{rth-footprint-lower-bound} 
by noticing the following. The set ${\rm in}_\prec(\mathcal{L}^*)$ 
of initial terms of $\mathcal{L}^*$ is equal to $\Delta_\prec(I)_{\leq d}$,
$\mathcal{N}_{\prec,r}$ is equal to $\mathcal{M}_{\prec,d,r}$, and
$\dim_K(\mathcal{L}_X)$ is equal to $|\Delta_\prec(I)_{\leq
d}|=H_I^a(d)$.
\end{proof}

\begin{remark}\label{sarabia-vila} Let $I=I(X)$ be the vanishing ideal of a
subset $X$ of $\mathbb{A}^s$. The
following hold. 
\begin{itemize}
\item[(a)] $r\leq\delta_r(C_X(d))\leq |X|$ 
for $d\geq 1$ and $1\leq r\leq H_I^a(d)$. This follows 
from the fact that the weight hierarchy is an
increasing sequence (see \cite[Theorem~1]{wei}).
\item[(b)] If $d\geq{\rm reg}(H_I^a)$, then
$C_X(d)=K^{|X|}$ and $\delta_r(C_X(d))=r$ for
$1\leq r\leq |X|$.
\item[(c)] If $C_X(d)$ is non-degenerate, i.e., for each
$1\leq i\leq |X|$ there is $\alpha\in C_X(d)$ whose
$i$-th entry is non-zero, then
$\delta_r(C_X(d))=|X|$ when $r=H_I^a(d)$. This follows from
Lemma~\ref{vila-delio-feb27-20} and 
Corollary~\ref{rth-GHW-affine} noticing that $V_X(F)=\emptyset$ for
$F\in\mathcal{F}_{\prec,d,r}$.
\end{itemize}
\end{remark}

\subsection*{Hilbert's Nullstellensatz over finite fields}
The next result is well-known, see \cite{Ghorpade} for an expository account of 
Nullstellensatz-type results like this. For
convenience we give a short proof of this result using our degree driven approach. 
We show examples of how to use this result to estimate the basic parameters
of evaluation codes over affine varieties defined by a given set of
polynomials of $S$ (Examples~\ref{Hermitian-curve-example} and 
\ref{elliptic-curve-example}). 

\begin{proposition}\label{Nullstellensatz-finite} Let $K$ be a finite
field $\mathbb{F}_q$, let $\mathbb{A}^s=K^s$ be the affine space
over $K$, and let $G$ be a finite set of polynomials of $S$. If
$X=V_{\mathbb{A}^s}(G)$ and
$(I(\mathbb{A}^s)\colon(G))\neq I(\mathbb{A}^s)$, then
$$
(I(\mathbb{A}^s),G)=(t_1^q-t_1,\ldots,t_s^q-t_s,G)=I(X).
$$
\end{proposition}

\begin{proof} As $I(\mathbb{A}^s)=(t_1^q-t_1,\ldots,t_s^q-t_s)$
\cite[p.~137]{JacI}, we need only show $(I(\mathbb{A}^s),G)=I(X)$. The
ideal $J:=(I(\mathbb{A}^s),G)$ is contained in $I(X)$ because
$G\subset I(V_{\mathbb{A}^s}(G))$. By
Lemma~\ref{degree-formula-zeros-affine} and
Proposition~\ref{behavior-hilbert-function}, we get 
$\deg(S/J)=|V_{\mathbb{A}^s}(G)|$ and $\deg(S/I(X)=|X|$, respectively.
Thus, $S/J$ and $S/I(X)$ have the same degree. Therefore, using 
the inclusion $J\subset I(X)$, by
additivity of the degree and Seidenberg's lemma the equality follows. 
\end{proof}

As a byproduct we obtain the next projective version over finite
fields of Hilbert's projective Nullstellensatz over algebraically closed
fields \cite[Theorem~A.4.6, p.~476]{singular}. This result 
can be used to estimate the parameters 
of evaluation codes over projective varieties defined by a
given set of homogeneous polynomials of $S$ 
\cite{min-dis-generalized,rth-footprint} (Example~\ref{Cox-example}). 

\begin{theorem}\label{Nullstellensatz-finite-projective} Let
$\mathbb{P}^{s-1}$ be a projective space over a finite field
$K=\mathbb{F}_q$, let $\mathbb{X}=V_{\mathbb{P}^{s-1}}(G)$ be the projective
variety defined by a finite set $G$ of homogeneous polynomials of
$S\setminus\{0\}$, and let $I=I(\mathbb{X})$ be its homogeneous vanishing ideal. If 
and $(I(\mathbb{P}^{s-1})\colon(G))\neq I(\mathbb{P}^{s-1})$, then
$$
{\rm rad}(I(\mathbb{P}^{s-1}),G)=
{\rm rad}(\{t_it_j^q-t_i^qt_j\vert\,1\leq i<j\leq s\},G)=I(\mathbb{X}).
$$
\end{theorem}

\begin{proof} As $I(\mathbb{P}^{s-1})=\{t_it_j^q-t_i^qt_j\vert\,1\leq
i<j\leq s\})$
\cite[Theorem~2.1]{mercier-rolland}, we need only show
that ${\rm rad}(I(\mathbb{P}^{s-1}),G)$ is equal to $I(\mathbb{X})$. The
ideal $J:=(I(\mathbb{P}^{s-1}),G)$ is a subset of $I(\mathbb{X})$ because
$G\subset I(V_{\mathbb{P}^{s-1}}(G))$. Thus, ${\rm rad}(J)$ 
is a subset of $I(\mathbb{X})$ because $I(\mathbb{X})$ is a radical
ideal. Using \cite[Lemma~3.4]{rth-footprint} and
Proposition~\ref{behavior-hilbert-function}, we get  
$$
|\mathbb{X}|=|V_{\mathbb{P}^{s-1}}(G)|=
\deg(S/(I(\mathbb{P}^{s-1}),G))=\deg(S/J)
$$
and  $\deg(S/I(\mathbb{X})=|\mathbb{X}|$, respectively.
Thus, $S/J$ and $S/I(\mathbb{X})$ have the same degree. Let
$[P_1],\ldots,[P_m]$ be the points of $\mathbb{X}$ and for each $i$ let
$\mathfrak{p}_i$ be the homogeneous vanishing ideal of the point $[P_i]$. 
Then, $I(\mathbb{X})=\bigcap_{i=1}^m\mathfrak{p}_i$. As $J$ and
$I(\mathbb{X})$ have height $s-1$ and $J\subset I(\mathbb{X})$, the
ideal $J$ has an irredundant primary decomposition of the form 
$J=\mathfrak{q}_1\bigcap\cdots\bigcap\mathfrak{q}_n\cap\mathfrak{q}$
such that $n\geq m$, $\mathfrak{q}_i$ is $\mathfrak{p}_i$-primary of
height $s-1$ for $i=1,\ldots,n$, and ${\rm
rad}(\mathfrak{q})=(t_1,\ldots,t_s)$. Hence, by additivity of the
degree Proposition~\ref{additivity-of-the-degree}, one has  
$$
|\mathbb{X}|=m=\deg(S/I(\mathbb{X}))=
\deg(S/J)=\sum_{i=1}^n\deg(S/\mathfrak{q}_i)\geq n.
$$
\quad Therefore, $n=m$ and consequently ${\rm rad}(J)=I(\mathbb{X})$.
\end{proof}

\section{Minimum distance of toric
codes}\label{toric-codes-section}
To avoid repetitions, we continue to employ
the notations and 
definitions used in Section~\ref{intro-section}.
In this section we determine the minimum distance of 
the toric code $\mathcal{C}_\mathcal{P}(d)$. Throughout this section
we assume that $K$ is a finite field $\mathbb{F}_q$.

\begin{lemma}{\rm(cf. \cite[Lemma~3.3]{min-dis-ci})}\label{apr5-19}
Let $d_1,\ldots,d_s$ be positive integers and let $L$ be the ideal of $S$ generated by
$t_1^{d_1},\ldots,t_{s}^{d_s}$. If $t^a=t_1^{a_1}\cdots t_s^{a_s}$ is
not in $L$, then
$$
\deg(S/(L,t^a))=d_1\cdots d_s-(d_1-a_1)\cdots(d_s-a_{s}).
$$
\end{lemma}
\begin{proof} The colon ideal $(L\colon t^a)$ is equal to
$(t_1^{d_1-a_1},\ldots,t_s^{d_s-a_s})$. Since $L$ and $(L\colon t^a)$
are complete intersections, one has $\deg(S/L)=\prod_{i=1}^sd_i$ and
$\deg(S/(L\colon t^a))=\prod_{i=1}^s(d_i-a_i)$. Taking affine Hilbert
functions in the exact sequence
$$
0\longrightarrow S/(L\colon t^a)[-e]\stackrel{t^a}{\longrightarrow}
S/L\longrightarrow S/(L,t^a)\longrightarrow 0,\quad e=\deg(t^a),
$$
we obtain $\deg(S/(L,t^a))=\deg(S/L)-\deg(S/(L\colon t^a))$. 
\end{proof}

A polynomial is called \textit{squarefree} if all its 
monomials are squarefree. The set of all squarefree monomials of $S$
of degree $d$ (resp. degree at most $d$) is denoted by $V_d$ 
(resp. $V_{\leq d}$). 

\begin{lemma}\label{toric-squarefree-standard}  
Let $\prec$ be a monomial order on $S$ and let 
$I=I(T)$ be the vanishing ideal of an affine torus
$T=(\mathbb{F}_q^*)^{s}$ over a finite field with $q\geq 3$ elements.
The following hold. 
\begin{enumerate}
\item[\rm(a)] The initial ideal ${\rm in}_\prec(I)$ of $I$ is generated by
$\{t_1^{q-1},\ldots,t_s^{q-1}\}$. In particular if $f$ is a
squarefree polynomial 
of $S$, then  $f$ is a standard polynomial of $S/I$. 
\item[\rm(b)] If $\mathcal{L}=KV_d$ or $\mathcal{L}=KV_{\leq d}$, then the
evaluation code $\mathcal{L}_T$ on $T$ is a standard evaluation code
on $T$ relative to $\prec$. 
\item[\rm(c)] The toric code $\mathcal{C}_\mathcal{P}(d)$ over the hypersimplex
$\mathcal{P}$ and the squarefree evaluation code $\mathcal{C}_{\leq
d}$ are standard evaluation codes on $T$ relative to $\prec$. 
\end{enumerate}
\end{lemma}
\begin{proof} (a): The ideal $I$ is generated by the set 
$\mathcal{B}=\{t_i^{q-1}-1\}_{i=1}^{s}$ and this set is a Gr\"obner
basis of $I$ (see \cite{GRH} and \cite[Lemma~2.3]{cartesian-codes} ).
Then, the initial ideal $L:={\rm in}_\prec(I)$ of $I$ is
generated by the set of monomials $\{t_i^{q-1}\}_{i=1}^{s}$. As
$q\geq 3$ and $f$ is squarefree, none of the terms of $f$ can be in
$L$. Thus, $f$ is a standard polynomial.

(b): By part (a), $V_d$ and $V_{\leq d}$ are subset of $K\Delta_\prec(I)$. 
Thus, $\mathcal{L}$ is a linear subspace of $K\Delta_\prec(I)$ and
$\mathcal{L}_T$ is a standard evaluation code.

(c): This follows from part (b) by recalling that $\mathcal{C}_\mathcal{P}(d)$ and 
$\mathcal{C}_{\leq d}$ are the images of $KV_d$ and $KV_{\leq d}$
under the evaluation map, respectively. 
\end{proof}

To prove the next proposition we use the results of
Section~\ref{degree-variety-section}.   
\begin{proposition}\label{footprint-technique} If $0\neq f\in
KV_d$, $q\geq 3$ and 
$1\leq d<s$, then 
$$
|V_T(f)|\leq(q-1)^{s}-(q-2)^d(q-1)^{s-d}.
$$
\end{proposition}
\begin{proof} Let $\prec$ be a monomial order on $S$ and let $I=I(T)$
be the vanishing ideal of the affine torus 
$T=(\mathbb{F}_q^*)^{s}$. 
By Theorem~\ref{affine-zeros-formula}, one has
\begin{equation}\label{apr17-1}
|V_T(f)|=\deg(S/(I,f)).
\end{equation}
\quad The initial ideal $L:={\rm
in}_\prec(I)$ of $I$ is
generated by the set $\{t_i^{q-1}\}_{i=1}^{s}$
(Lemma~\ref{toric-squarefree-standard}). Let
$t^a={\rm in}_\prec(f)=t_1^{a_1}\cdots 
t_s^{a_s}$ be the initial monomial of $f$.  Since $f$ is squarefree, so is
$t^a$. As $q\geq 3$, $t^a$ cannot be in $L$. Therefore, by
Theorem~\ref{degree-initial-footprint-affine} and 
Lemma~\ref{apr5-19}, we get  
\begin{equation}\label{apr17-2}
\deg(S/(I,f))\leq(q-1)^{s}-(q-2)^{d}(q-1)^{s-d},
\end{equation}
where $d=\deg(f)$. Thus, the inequality follows at once from
Eq.~(\ref{apr17-1}).
\end{proof}

\begin{proposition}\label{may6-19} 
Let $\mathcal{C}_\mathcal{P}(d)$ be the toric code of $\mathcal{P}$
of degree $d$. Then, the length of $\mathcal{C}_\mathcal{P}(d)$ is
equal to $(q-1)^s$ and its dimension is given by
$$
\dim_K(\mathcal{C}_\mathcal{P}(d))=\begin{cases}
\binom{s}{d}&\mbox{ if }q\geq 3,\\
\ 1&\mbox{ if }q=2.
\end{cases}
$$
\end{proposition}

\begin{proof} The length $m$ of $\mathcal{C}_\mathcal{P}(d)$ is the 
number of points of $T$, that is, $m=(q-1)^s$. 
Assume that $q\geq 3$. The number of squarefree monomials 
of $S$ of degree $d$ is $\binom{s}{d}$. Then, one has
$\dim_K(KV_d)=\binom{s}{d}$. Hence, it suffices to note that, by 
Lemmas~\ref{apr26-20} and \ref{toric-squarefree-standard}, 
the evaluation map gives an isomorphism between 
$KV_d$ and $\mathcal{C}_\mathcal{P}(d)$.

Assume that $q=2$. Then, $T=\{(1,\ldots,1)\}$, $m=1$, 
$\mathcal{C}_\mathcal{P}(d)=\mathbb{F}_2$, and $\dim_K(\mathcal{C}_\mathcal{P}(d))=1$.
\end{proof}
We come to the main result of this section.
\begin{theorem}\label{Notre-Dame-Cathedral-Apr15-2019} 
Let $\mathcal{C}_\mathcal{P}(d)$ be the toric code of $\mathcal{P}$
of degree $d$ and let $\delta(\mathcal{C}_\mathcal{P}(d))$
be its minimum distance. Then 
$$
\delta(\mathcal{C}_\mathcal{P}(d))=\begin{cases}
(q-2)^d(q-1)^{s-d}&\mbox{ if }d\leq s/2,\, q\geq 3,\\
(q-2)^{s-d}(q-1)^{d}&\mbox{ if }s/2 < d < s, \,q\geq 3,\\
(q-1)^{s} &\mbox{ if } d=s,
\\
1 &\mbox{ if } q=2.
\end{cases}
$$
\end{theorem}

\begin{proof} Assume that $s\geq 2d$ and $q\geq 3$. We set 
$\eta(d)=(q-2)^d(q-1)^{s-d}$ and
$\phi(d)=(q-1)^{s}-\eta(d)$. Let $T=(K^*)^s$ be the affine torus
in $\mathbb{A}^s$, let $\prec$ be a monomial order on $S$, and let
$\mathcal{L}:=KV_d$ be the linear space 
generated by $V_d$. By Lemma~\ref{toric-squarefree-standard},
$\mathcal{L}_X=\mathcal{C}_\mathcal{P}(d)$ is a standard evaluation
code. 
Then, by Theorems~\ref{affine-zeros-formula} and
\ref{GHW-standard-code}, there 
is $0\neq f\in\mathcal{L}$ such that
\begin{equation}\label{jun14-19}
\delta(\mathcal{C}_\mathcal{P}(d))=\min\{|T\setminus V_T(g)|\colon
0\neq g\in\mathcal{L}\}=|T\setminus
V_T(f)|=(q-1)^{s}-|V_T(f)|.
\end{equation}
\quad Thus, by Proposition~\ref{footprint-technique}, one has
$|V_T(f)|\leq\phi(d)$ and $\delta(\mathcal{C}_\mathcal{P}(d))\geq
\eta(d)$. Consider the 
squarefree homogeneous polynomial of degree $d$
$$
f_d=h_1\cdots h_d=(t_1-t_2)\cdots(t_{2d-1}-t_{2d}),
$$
where $h_i=t_{2i-1}-t_{2i}$ for $i=1,\ldots,d$. By
Eq.~(\ref{jun14-19}), to prove the inequality 
$\delta(\mathcal{C}_\mathcal{P}(d))\leq\eta(d)$ it suffices to show that the polynomial $f_d$ has
exactly $\phi(d)$ roots in $T$. As 
$V_T(f_d)$ is equal to $\bigcup_{i=1}^dV_T(h_i)$, using the inclusion-exclusion 
principle \cite[p.~38, Formula~2.12]{aigner}, we get 
\begin{eqnarray*}
|V_T(f_d)|&=&\sum_{1\leq \ell_1\leq
d}|V_T(h_{\ell_1})|-\sum_{1\leq \ell_1<\ell_2\leq
d}|V_T(h_{\ell_1})\textstyle\bigcap V_T(h_{\ell_2})|\\
& &\ \ \ \ \ \ \ \ +\cdots+(-1)^{d-1}|V_T(h_1)
\textstyle\bigcap \cdots
\textstyle\bigcap  V_T(h_d)|.
\end{eqnarray*}
\quad The variables occurring in $h_i$ and $h_j$ are disjoint for $i\neq j$. 
Thus, counting monomials in each of the intersections, one obtains
$$
\sum_{1\leq\ell_1<\cdots<\ell_i\leq
d}|V_T(h_{\ell_1})\textstyle\bigcap \cdots\textstyle\bigcap  V_T(h_{\ell_i})|=
\binom{d}{i}(q-1)^{s-i},
$$
and consequently the number of zeros of $f_d$ in $T$ is given by
\begin{eqnarray}\label{Notre-Dame-Cathedral}
|V_T(f_d)|&=&\sum_{i=1}^d(-1)^{i-1}\binom{d}{i}(q-1)^{s-i}=(q-1)^{s-d}
\sum_{i=1}^d(-1)^{i-1}\binom{d}{i}(q-1)^{d-i}\nonumber\\
&=&(q-1)^{s-d}\left[(q-1)^d-((q-1)-1)^d\right]\nonumber\\
&=&
(q-1)^{s-d}\left[(q-1)^d-(q-2)^d\right]\nonumber\\
&=&(q-1)^s-(q-2)^d(q-1)^{s-d}=\phi(d).
\end{eqnarray}

Assume that $s<2d$, $d<s$ and $q\geq 3$. The affine torus
$T$ in $\mathbb{A}^{s}$ is a group under
componentwise multiplication and the map
$\sigma\colon{T}\rightarrow{T}$,
$[(x_1,\ldots,x_s)]\mapsto [(x_1^{-1},\ldots,x_s^{-1})]$ is
a group isomorphism. Setting $Q_i:=\sigma(P_i)$ for $i=1,\ldots,m$, we can write 
$$
{T}=\{P_1,\ldots,P_m\}=\{Q_1,\ldots,Q_m\}.
$$
\quad We denote the toric code of $\mathcal{P}$ of degree
$d$ with respect to $\{P_1,\ldots,P_m\}$ by $\mathcal{C}_d$ and 
denote the toric code of 
$$
\mathcal{Q}:={\rm
conv}(\{\mathbf{e}_{i_1}+\cdots+\mathbf{e}_{i_{s-d}}\vert\, 1\leq
i_1<\cdots< i_{s-d}\leq s\})
$$
of degree $s-d$ with respect to $\{Q_1,\ldots,Q_m\}$ by
$\mathcal{C}_{s-d}$. Thus,
\begin{eqnarray*}
\mathcal{C}_d=\{(f(P_1),\ldots,f(P_m))\vert\, f\in KV_d\},
\ \ \ \ \ \ \ \ \ \ \ 
& &\\
\mathcal{C}_{s-d}=\{(g(Q_1),\ldots,g(Q_m))\vert\,
g\in KV_{s-d}\}.& &
\end{eqnarray*}
\quad The basic parameters of the projective toric codes 
$\mathcal{C}_d$ and $\mathcal{C}_{s-d}$
do not depend on how we order the elements of $T$. For use 
below, let $\{M_1,\ldots,M_n\}$ be the set of all squarefree monomials
of  $S$ of degree $d$. Setting $M_i^c:=t_1\cdots t_s/M_i$ for
$i=1,\ldots,n$, note that $\{M_1^c,\ldots,M_n^c\}$ is the set of all
squarefree monomials of $S$ of degree $s-d$. If $f\in KV_d$, 
writing $f=\sum_{i=1}^n\lambda_iM_i$, $\lambda_i\in\mathbb{F}_q$ for all $i$, 
we set $f^c:=\sum_{i=1}^n\lambda_iM_i^c$. From the equalities 
\begin{eqnarray*}
t_1\cdots t_sf^c\left(\frac{1}{t_1},\ldots,\frac{1}{t_s}\right)
&=&t_1\cdots
t_s\sum_{i=1}^n\lambda_iM_i^c\left(\frac{1}{t_1},\ldots,\frac{1}{t_s}\right)\\
&=&
\sum_{i=1}^n\lambda_iM_i(t_1,\ldots,t_s)=f(t_1,\ldots,t_s),
\end{eqnarray*}
we get $t_1\cdots t_sf^c\left({t_1^{-1}},\ldots,{t_s^{-1}}\right)
=f(t_1,\ldots,t_s)$. Hence, setting $P_k=(p_{k,1},\ldots,p_{k,s})$,
one has 
\begin{equation}\label{feb11-20}
p_{k,1}\cdots p_{k,s}f^c(p_{k,1}^{-1},\ldots,
p_{k,s}^{-1})=p_{k,1}\cdots p_{k,s}f^c(Q_k)=f(P_k) 
\end{equation}
for $k=1,\ldots,m$. There is a commutative diagram
$$
\begin{array}{ccc}
\begin{array}{ccc}
KV_d
&\stackrel{\mathrm{ev}_d\ }{\longrightarrow}& 
\mathcal{C}_{d}\\ 
\Big\downarrow\rlap{$\psi$}& &\Big\downarrow\rlap{$\psi'$}\\ 
KV_{s-d}
&\stackrel{\mathrm{ev}_{s-d}}{\longrightarrow}& 
\mathcal{C}_{s-d}
\end{array}
&
\quad
&
\begin{array}{ccc}
f &{\longrightarrow}& 
(f(P_1),\ldots,f(P_m))\\ 
\Big\downarrow& &\Big\downarrow\\ 
f^{c}
&{\longrightarrow}& 
(f^c(Q_1),\ldots,f^c(Q_m))
\end{array}
\end{array}
$$
where all the maps are isomorphisms of linear spaces. By
Eq.~(\ref{feb11-20}), $f(P_k)=0$ if and only if $f^c(Q_k)=0$. Hence,
the linear codes $\mathcal{C}_d$ and $\mathcal{C}_{s-d}$ have the  
same minimum distance. Since $s-d\leq s/2$, by the previous part, 
it follows that 
$$
\delta(\mathcal{C}_\mathcal{P}(s-d)):=\delta(\mathcal{C}_{s-d})=
(q-2)^{s-d}(q-1)^{s-(s-d)}=(q-2)^{s-d}(q-1)^{d}.
$$
\quad Thus, the minimum distance of $\mathcal{C}_d$ is
equal to $(q-2)^{s-d}(q-1)^{d}$, as required.

Assume that $d=s$. Then, $KV_d$ is generated by $f:=t_1\cdots
t_d$ and the toric code $\mathcal{C}_\mathcal{P}(d)$ is generated by
$(f(P_1),\ldots,f(P_m))$. Since $f(P_i)\neq 0$ for all $i$, one has
$\delta(\mathcal{C}_\mathcal{P}(d))=m=(q-1)^{s}$.

Assume that $q=2$. Then, ${T}=\{(1,\ldots,1)\}$, $m=1$, and
$\mathcal{C}_\mathcal{P}(d)=\mathbb{F}_2$. Thus,
$\delta(\mathcal{C}_\mathcal{P}(d))=1$. 
\end{proof}

\begin{corollary}\label{case-s-2d} Let $f$ be a squarefree 
homogeneous polynomial in $S_d\setminus\mathbb{F}_q$ and let $T$
be the affine torus $(\mathbb{F}_q^*)^s$ of $\mathbb{A}^s$. If $q\geq
3$ and $s/2 < d < s$, then 
$$
|V_T(f)|\leq(q-1)^{s}-(q-2)^{s-d}(q-1)^{d},
$$
and this upper bound is sharp.
\end{corollary}

\demo  Let $\delta(\mathcal{C}_\mathcal{P}(d))$ be the minimum distance of
$\mathcal{C}_\mathcal{P}(d)$. By
Theorem~\ref{Notre-Dame-Cathedral-Apr15-2019}, one has
\begin{eqnarray*}
\delta(\mathcal{C}_\mathcal{P}(d))&=&\min\{|{T}\setminus V_{T}(g)|\colon
0\neq g\in KV_d\}\\ 
&=&|{T}|-\max\{|V_{T}(g)|\colon
0\neq g\in KV_d\}
=(q-2)^{s-d}(q-1)^{d}.
\end{eqnarray*}
\quad Therefore, $|V_T(f)|\leq(q-1)^s-(q-2)^{s-d}(q-1)^d$ 
and this upper bound is sharp. For convenience we construct a polynomial
in $KV_d$ where equality occurs.  There is $1\leq k\leq d-1$ such
that $s=d+k$. Consider the following 
squarefree homogeneous polynomial of degree $d$
$$
g_k:=h_1\cdots h_kt_{2k+1}\cdots
t_{k+d}=(t_1-t_2)\cdots(t_{2k-1}-t_{2k})t_{2k+1}\cdots t_{k+d},
$$
where $h_i=t_{2i-1}-t_{2i}$ for $i=1,\ldots,k$. Since 
$V_T(g_k)$ is equal to $V_T(f_k)=\bigcup_{i=1}^kV_T(h_i)$,
where $f_k:=h_1\cdots h_k$, and $k=s-d$, using
Eq.~(\ref{Notre-Dame-Cathedral}) with $d\rightarrow k=s-d$ and
noticing $s>2k$, we get 
\begin{eqnarray*}
|V_T(g_k)|&=&|V_T(h_1\cdots h_k)|=|V_T(f_k)|\\
&=&(q-1)^s-(q-2)^k(q-1)^{s-k}\\
&=&(q-1)^s-(q-2)^{s-d}(q-1)^{d}.\quad \Box
\end{eqnarray*}
\quad The minimum distance $\delta(C_{T}(d))$ of the affine 
Reed--Muller-type code $C_T(d)$ is non-increasing as a
function of $d$ (Proposition~\ref{behavior-hilbert-function}). This
is no longer 
the case for the minimum distance of the toric code
$\mathcal{C}_\mathcal{P}(d)$ (Example~\ref{example-ptc}).

\section{Squarefree affine evaluation
codes}\label{squarfree-codes-section}

In this section we determine the minimum distance and the 2nd
generalized Hamming weight of a squarefree evaluation code.  
To avoid repetitions, we continue to employ
the notations and definitions used in Sections~\ref{intro-section} 
and \ref{GHW-section}. 

\begin{proposition}\label{squarefree-affine} Let $f$ be a squarefree 
polynomial in $S\setminus\mathbb{F}_q$ of degree at most $d$ and let $T$
be the affine torus $(\mathbb{F}_q^*)^s$ of $\mathbb{A}^s$. If $q\geq
3$ and $d\leq s$, then 
$$
|V_T(f)|\leq(q-1)^{s}-(q-2)^{d}(q-1)^{s-d},
$$
with equality if $d\geq 1$ and $f=(t_1-1)\cdots(t_d-1)$.
\end{proposition}

\begin{proof} Let $I=I(T)$ be the vanishing ideal of $T$. If $f$ is
not a zero divisor of $S/I$, then $(I\colon f)=I$ and, 
by Lemma~\ref{degree-formula-zeros-affine}, the set 
$V_T(f)$ is empty and the required inequality is clear. 
Thus, we may assume that $f$ is a zero-divisor of $S/I$. In particular
$d\geq 1$. If $e=\deg(f)$, then
$$
(q-1)^s-(q-2)^e(q-1)^{s-e}\leq (q-1)^s-(q-2)^d(q-1)^{s-d}.
$$
\quad Thus, we may also assume that $d=\deg(f)$. Let $\prec$ be a
graded monomial order and let $t^a={\rm in}_\prec(f)$ be the initial
monomial of
$f$. Note that $t^a$ is squarefree and $d=\deg(t^a)$ since $\prec$ is
graded. The initial ideal $L:={\rm in}_\prec(I)$ of $I$ is
generated by the set of monomials $\{t_i^{q-1}\}_{i=1}^{s}$
(Lemma~\ref{toric-squarefree-standard}). As $q\geq 3$,
$t^a$ cannot be in $L$. Hence, by
Theorem~\ref{degree-initial-footprint-affine} and Lemma~\ref{apr5-19}, we get 
\begin{equation*}\label{feb13-20}
|V_T(f)|=\deg(S/(I,f))\leq \deg(S/(L,t^a)=(q-1)^{s}-(q-2)^{d}(q-1)^{s-d}.
\end{equation*}
\quad Now, assume that
$f=h_1\cdots h_d$, where $h_i=t_i-1$ for $i=1,\ldots,d$. As in the
proof of Theorem~\ref{Notre-Dame-Cathedral-Apr15-2019}, using the inclusion-exclusion 
principle \cite[p.~38, Formula~2.12]{aigner}, the formula for $|V_T(f)|$
follows readily. 
\end{proof}

\begin{proposition}\label{feb28-20-1} 
Let $\mathcal{C}_{\leq d}$ be the squarefree evaluation code of degree $d$ on the
affine torus $T=(\mathbb{F}_q^*)^s$. Then, the length of
$\mathcal{C}_{\leq d}$ is $(q-1)^s$, and the dimension 
of $\mathcal{C}_{\leq d}$ is given by
$$
\dim_K(\mathcal{C}_{\leq d})=\begin{cases}
\binom{s}{0}+\binom{s}{1}+\cdots+\binom{s}{d}&\mbox{ if }q\geq 3,\\
\ 1&\mbox{ if }q=2.
\end{cases}
$$
\end{proposition}

\begin{proof} The length $m$ of the code $\mathcal{C}_{\leq d}$ is the number of
points of $T$, that is, $m=(q-1)^s$. 
Assume that $q\geq 3$. The number of squarefree monomial
of $S$ of degree at most $d$ is equal to 
$k:=\sum_{i=0}^d\binom{s}{i}$. Then, one has
$k=\dim_K(KV_{\leq d})$. Setting $\mathcal{L}=KV_{\leq d}$, 
by Lemmas~\ref{apr26-20} and \ref{toric-squarefree-standard}, one 
has $\mathcal{L}\simeq\mathcal{L}_T=\mathcal{C}_{\leq d}$. Thus,
$\dim_K(\mathcal{C}_{\leq d})=k$.

Assume that $q=2$. Then $T=\{(1,\ldots,1)\}$, $m=1$, 
$\mathcal{C}_{\leq d}=\mathbb{F}_2$, and $\dim_K(\mathcal{C}_{\leq d})=1$.
\end{proof}

Let $\prec$ be a monomial order and  let $\mathcal{H}_{\prec,d,r}$ be
the set of all $F=\{f_1,\ldots,f_r\}\subset (KV_{\leq d})^*$ such 
that ${\rm in}_\prec(f_1),\ldots,{\rm in}_\prec(f_r)$ are distinct
monomials and $f_i$ is monic for all $i$.

\begin{corollary}\label{rth-GHW-squarefree}
Let $K=\mathbb{F}_q$ be a finite field and let $I=I(T)$ be the vanishing ideal of the
affine torus $T=(\mathbb{F}_q^*)^s$. If
$q\geq 3$ and $\prec$ is a graded monomial order, then  
$$
\delta_r(\mathcal{C}_{\leq d})=|T|-\max\{\deg(S/(I,F))\vert\,
F\in\mathcal{H}_{\prec, d,r}\}\mbox{ for }d\geq 1
\mbox{ and }1\leq r\leq\dim_K(C_{\leq d}).
$$
\end{corollary}
\begin{proof}
We set $\mathcal{L}=KV_{\leq d}$. By
Lemma~\ref{toric-squarefree-standard}, $\mathcal{L}_T$ is a standard
evaluation code on $T$ relative to $\prec$. Then, $\mathcal{C}_{\leq
d}=\mathcal{L}_X$, and the formula for
$\delta_r(\mathcal{C}_{\leq d})$ is a direct consequence of    
Theorem~\ref{GHW-standard-code}.
\end{proof}

To show a lower bound for $\delta_r(\mathcal{C}_{\leq d})$,
let $\mathcal{J}_{\prec, d, r}$ be the 
family of all sets $M=\{t^{a_1},\ldots,t^{a_r}\}$ such that
$t^{a_1},\ldots,t^{a_r}$ are distinct squarefree monomials of 
$S_{\leq d}$. 
The $r$-th \textit{squarefree footprint} of
$\mathcal{C}_{\leq d}$ of
degree $d$, 
denoted $\rho_I(d,r)$, is given
by 
$$
\rho_I(d,r):=\deg(S/I)-\max\{\deg(S/({\rm in}_\prec(I),M))\,\vert\,
M\in\mathcal{J}_{\prec, d,r}\}.
$$

\begin{corollary}\label{rth-footprint-lower-bound-squarefree}  
Let $K=\mathbb{F}_q$ be a finite field and let $I=I(T)$ be the vanishing ideal of the
affine torus $T=(\mathbb{F}_q^*)^s$. If
$q\geq 3$ and $\prec$ is a graded monomial order, then
$$\rho_{I}(d,r)\leq \delta_r(\mathcal{C}_{\leq d})\ \mbox{ for
}d\geq 1\mbox{ and }1\leq r\leq \dim_K(\mathcal{C}_{\leq d}).$$
\end{corollary}
\begin{proof}
We set $\mathcal{L}=KV_{\leq d}$. By
Lemma~\ref{toric-squarefree-standard}, $\mathcal{L}_T$ is a standard
evaluation code on $T$ relative to $\prec$. Then, $\mathcal{C}_{\leq
d}=\mathcal{L}_X$, and the lower bound for
$\delta_r(\mathcal{C}_{\leq d})$ is a direct consequence of    
Theorem~\ref{rth-footprint-lower-bound}.
\end{proof}
We come to one of the main results of this section.

\begin{theorem}\label{min-dis-squarefree} Let $\mathcal{C}_{\leq d}$
be the squarefree evaluation code of degree $d$ on the
affine torus $T=(\mathbb{F}_q^*)^s$. If $q\geq 3$, then the minimum
distance $\delta(\mathcal{C}_{\leq d})$
of $\mathcal{C}_{\leq d}$ is $(q-2)^{d}(q-1)^{s-d}$.
\end{theorem}

\begin{proof} Let $\prec$ be a monomial order on $S$. We set
$\mathcal{L}=KV_{\leq d}$. 
As $q\geq 3$, by Lemma~\ref{toric-squarefree-standard},
$\mathcal{L}_T$ is a standard evaluation code on $T$ relative to $\prec$. 
Then, by Corollary~\ref{rth-GHW-squarefree}, we get 
$$
\delta(\mathcal{C}_{\leq d})=\min\{|T\setminus V_T(g)|\colon
g\in\mathcal{L}^*\},
$$ 
and the result follows from Proposition~\ref{squarefree-affine}.
\end{proof}

We come to another of the main results of this section. 

\begin{theorem}\label{2nd-GHW-squarefree}
If $q\geq 3$ and $d\geq 1$, then the second generalized Hamming
weight of $\mathcal{C}_{\leq d}$ is 
$$
\delta_2(\mathcal{C}_{\leq d})=
\begin{cases}
(q-2)^{s-1}(q-1)&\mbox{ if }\, d=s,\\
(q-2)^d(q-1)^{s-d-1}q&\mbox{ if }\, d<s.
\end{cases}
$$
\end{theorem}

\demo Let $\prec$ be a graded monomial order and let $I=I(T)$
be the vanishing ideal of $T$. The initial ideal of $I$ is
$L=(\{t_i^{q-1}\}_{i=1}^{s})$
(Lemma~\ref{toric-squarefree-standard}). If $s=1$, then $d=s=1$ and
$\delta_2(\mathcal{C}_{\leq 1})=q-1$. Indeed, take $\{f_1,f_2\}$ in
$\mathcal{H}_{\prec,1,2}$ and notice that $(f_1,f_2)=S$, then 
$\deg(S/(t_1^{q-1}-1,f_1,f_2))=0$ and consequently, by
Corollary~\ref{rth-GHW-squarefree}, $\delta_2(\mathcal{C}_{\leq
1})=q-1$. Thus, we may assume $s\geq 2$.  

The support of a monomial 
$t^a$, denoted ${\rm supp}(t^a)$, is the set of
all $t_i$ that occur in $t^a$. Take
$\{t^a,t^b\}\in\mathcal{J}_{\prec,d,2}$, $t^a=\prod_{i=1}^st_i^{a_i}$
and $t^b=\prod_{i=1}^st_i^{b_i}$. We set $e=\deg(t^a)$,
$n=\deg(t^b)$, $A={\rm supp}(t^a)$, and $B={\rm supp}(t^b)$. 
We may assume $e\leq n\leq d$. As $t^a$ and $t^b$ are squarefree
monomials, one
has
\begin{equation*}
\prod_{i=1}^s\min\{q-1-a_i,q-1-b_i\}=
(q-2)^{|A\cup B|}(q-1)^{s-|A\cup B|}.
\end{equation*}
\quad Therefore, picking a new variable $u$
and setting $J=(L,t^a,t^b)$, from \cite[p.~343]{rth-footprint} and
using that $t^a,t^b$ are squarefree 
monomials, we obtain
\begin{eqnarray}\label{apr1-20}
\deg(S/J)=\deg(S[u]/JS[u])=(q-1)^s-(q-2)^e(q-1)^{s-e}\quad\quad\quad\quad& &\\
\quad\quad\quad -(q-2)^n(q-1)^{s-n}+
(q-2)^{|A\cup B|}(q-1)^{s-|A\cup B|}.& &\nonumber
\end{eqnarray}   

Case (a): Assume $d=s$. As $t^a,t^b$ are distinct squarefree
monomials, one has $e<s$. Indeed, if $e=s$, then $t^a=t_1\cdots t_s$,
$s=e\leq n\leq s$, and $t^a=t^b$, a contradiction. First we show the 
inequality $\delta_2(\mathcal{C}_{\leq d})\geq (q-2)^{s-1}(q-1)$. By
Corollary~\ref{rth-footprint-lower-bound-squarefree},
$\delta_2(\mathcal{C}_{\leq d})\geq\rho_{I}(d,2)$. Thus, it suffices to 
show the inequality $\rho_I(d,2)\geq(q-2)^{s-1}(q-1)$. Hence, we need
only show
\begin{equation*}
(q-1)^s-(q-2)^{s-1}(q-1)\geq \deg(S/J)=\deg(S/(L,t^a,t^b))
\end{equation*}   
for any $\{t^a,t^b\}$ in $\mathcal{J}_{\prec,d,2}$. Note that, by
Eq.~(\ref{apr1-20}), this inequality is equivalent to 
\begin{eqnarray*}
(q-2)^e(q-1)^{s-e}+(q-2)^n(q-1)^{s-n}\geq\quad\quad\quad\quad\quad\quad\quad& &\\ 
(q-2)^{s-1}(q-1)+(q-2)^{|A\cup B|}(q-1)^{s-|A\cup B|}\,.& & 
\end{eqnarray*}   
\quad That this inequality holds follows recalling that $e<s$
and noticing the next two inequalities
\begin{eqnarray*}
(q-2)^e(q-1)^{s-e}\geq(q-2)^{s-1}(q-1),\quad\quad\quad\quad\quad\quad\quad\quad& &\\
(q-2)^n(q-1)^{s-n}\geq 
(q-2)^{|A\cup B|}(q-1)^{s-|A\cup B|}\,.&&
\end{eqnarray*}
\quad Now, we show the inequality $\delta_2(\mathcal{C}_{\leq d})\leq
(q-2)^{s-1}(q-1)$. By Corollary~\ref{rth-GHW-squarefree} and
Theorem~\ref{affine-zeros-formula}, it suffices to find $\{f_1,f_2\}$ in
$\mathcal{H}_{\prec,s,2}$ such that 
$$\deg(S/(I,f_1,f_2))=|V_T(f_1,f_2)|=|V_T(f_1)\cap
V_T(f_2)|=(q-1)^s-(q-2)^{s-1}(q-1).$$ 
\quad Setting $f_1=(t_1-1)\cdots(t_s-1)$ and
$f_2=(t_2-1)\cdots(t_s-1)$, one has $|V_T(f_1,f_2)|=|V_T(f_2)|$. 
As $\deg(f_2)=d-1$ and $d=s$, by Proposition~\ref{squarefree-affine},
we get $$|V_T(f_2)|=(q-1)^s-(q-2)^{s-1}(q-1).$$ 
\quad Case (b): Assume $d<s$. First we show the
inequality $\delta_2(\mathcal{C}_{\leq d})\geq (q-2)^d(q-1)^{s-d-1}q$. By
Corollary~\ref{rth-footprint-lower-bound-squarefree},
$\delta_2(\mathcal{C}_{\leq d})\geq\rho_{I}(d,2)$. Thus, it suffices to 
show $\rho_I(d,2)\geq (q-2)^d(q-1)^{s-d-1}q$. Hence, we need
only show that the inequality
\begin{equation*}
(q-1)^s-(q-2)^d(q-1)^{s-d-1}q\geq \deg(S/J)
\end{equation*}   
holds for any $\{t^a,t^b\}$ in $\mathcal{J}_{\prec,d,2}$. Note that,
by Eq.~(\ref{apr1-20}), this inequality is equivalent to 
\begin{eqnarray}\label{apr1-20-1}
(q-2)^e(q-1)^{s-e}+(q-2)^n(q-1)^{s-n}\geq\quad\quad\quad\quad\quad\quad\quad & &\\ 
(q-2)^d(q-1)^{s-d-1}q+(q-2)^{|A\cup B|}(q-1)^{s-|A\cup B|}.& &\nonumber 
\end{eqnarray}   
\quad To show this inequality we consider two subcases.

($\mathrm{b}_1$): Assume $e=d$. Then, $e=n=d$, and
Eq.~(\ref{apr1-20-1}) is equivalent to
$$
(q-2)^{d+1}(q-1)^{s-d-1}\geq 
(q-2)^{|A\cup B|}(q-1)^{s-|A\cup B|}. 
$$
\quad This inequality follows by noticing that $|A\cup B|\geq d+1$. Indeed, if $|A\cup B|=d$, then
$A=B$ and $t^a=t^b$, a contradiction. 

($\mathrm{b}_2$): Assume $e<d$. Setting $k=d-e$, one has
$(q-1)^{k+1}\geq(q-2)^kq$ because $k\geq 1$. This inequality is easy
to show using induction on $k\geq 1$. That
Eq.~(\ref{apr1-20-1}) holds true now follows from the following two inequalities
\begin{eqnarray*}
(q-2)^e(q-1)^{s-e}\geq(q-2)^d(q-1)^{s-d-1}q,\quad\quad\quad\quad& &\\ 
(q-2)^n(q-1)^{s-n}\geq(q-2)^{|A\cup B|}(q-1)^{s-|A\cup B|}.&& 
\end{eqnarray*}
\quad Finally, we show the inequality $\delta_2(\mathcal{C}_{\leq d})\leq
(q-2)^d(q-1)^{s-d-1}q$. By Corollary~\ref{rth-GHW-squarefree} and
Theorem~\ref{affine-zeros-formula}, it suffices to find $\{f_1,f_2\}$ in
$\mathcal{H}_{\prec,d,2}$ such that 
$$\deg(S/(I,f_1,f_2))=|V_T(f_1,f_2)|=|V_T(f_1)\cap
V_T(f_2)|=(q-1)^s-(q-2)^d(q-1)^{s-d-1}q.$$ 
\quad Setting $f_1=(t_1-1)\cdots(t_d-1)$, 
$f_2=(t_2-1)\cdots(t_{d+1}-1)$, and $g=(t_1-1)\cdots(t_{d+1}-1)$, one has $|V_T(f_1)\cup
V_T(f_2))|=|V_T(g)|$ and, 
by Proposition~\ref{squarefree-affine}, we get
\begin{eqnarray*}
|V_T(f_1,f_2)|&=&|V_T(f_1)|+|V_T(f_2)|-|V_T(g)|\\
&=&2((q-1)^s-(q-2)^d(q-1)^{s-d})-((q-1)^s-(q-2)^{d+1}(q-1)^{s-d-1})\\
&=&(q-1)^s-(q-2)^d(q-1)^{s-d-1}q.\quad\Box
\end{eqnarray*}

\section{Examples}\label{examples-section}
This section includes examples illustrating some of our
results. In Appendix~\ref{Appendix} we give the implementations in
\textit{Macaulay}$2$ \cite{mac2} and \textit{Magma}
\cite{magma} that are used in the examples.

\begin{example}\label{transforming-example} 
Let $K$ be the field $\mathbb{F}_5$, let $S=K[t_1,t_2]$ be a
polynomial ring in two variables, and let $\mathcal{L}$ be the
$K$-linear space generated by all monomials $t_1^{a_1}t_2^{a_2}$ such
that $(a_1,a_2)$ is one of the solid points of the point configuration
depicted on the left of Figure~\ref{transforming-figure}. The
evaluation code $\mathcal{L}_T$, over the torus
$T=(\mathbb{F}_5^*)^2$, is a generalized toric code in the sense 
of \cite{Little,Ruano}. Let $\prec$ be a monomial order. 
The vanishing ideal $I$ of $T$ is generated by
the Gr\"obner basis $G=\{t_1^4-1,t_2^4-1\}$. 
Then, $\mathcal{L}$ is generated by 
$$ 
B=\{1,\, t_1^3,\, t_1t_2^2,\, t_2^3,\, t_1t_2,\, t_1^2t_2^4\},
$$
and the list of remainder of the elements of $B$ on division by
$G$ is
$$
\widetilde{B}=\{1,\, t_1^3,\, t_1t_2^2,\, t_2^3,\, t_1t_2,\, t_1^2\},
$$
that is, $\widetilde{\mathcal{L}}$ is generated by 
$\widetilde{B}$ and $\widetilde{\mathcal{L}}_T$ is a standard evaluation
code on $T$ relative to $\prec$. Using Theorem~\ref{GHW-standard-code}, 
Proposition~\ref{transforming-to-standard-form}, 
Theorem~\ref{rth-footprint-lower-bound}, and
Procedure~\ref{transforming-procedure}, we obtain that the minimum
distance $\delta_1(\mathcal{L}_T)$ of ${\mathcal{L}}_T$ is $8$ and the
footprint ${\rm fp}_1(\mathcal{L}_X)$ is $4$. The length and the
dimension of $\mathcal{L}_X$ are $16$ and $6$, respectively.
\begin{figure}[H]
\begin{tabular}{ccc}
\begin{tikzpicture}[scale = 0.75]
\begin{axis}[axis lines=center, enlargelimits=false, 
 xmin=-0.5,
  xmax=4,
  ymin=-0.5,
  ymax=4,
  title={}]
  \addplot [only marks, mark size=2.5] table {
  0 0
  0 3
  1  1
  1  2
  2  4
  3  0
};
\addplot [only marks, mark=o, mark size=2.5] table {
 1  3
 2  3
 2  2
 2  1
 0  1
 0  2
 1  0
 2  0
};
\addplot [domain=0:2, samples=2]{(1/2)*x+3};
\addplot [domain=2:3, samples=2]{-4*x+12};
 \end{axis}
\end{tikzpicture}
\xygraph{
 !{<0cm,.5cm>;<.5cm,0cm>;<0cm,.5cm>} 
 !{(0,2)}*{>}="v3"
 !{(-1,.3)}*{}="v2"
 !{(-1,2)}*{}="v1"
 !{(-1.3,2)}*{\text{{\Large$\mathcal{L}$}}}
 !{(0.4,2)}*{\text{{\Large$\tilde{\mathcal{L}}$}}}
 "v1"-"v3"
}&&\begin{tikzpicture}[scale = 0.75]
\begin{axis}[axis lines=center, enlargelimits=false, 
 xmin=-0.5,
  xmax=4,
  ymin=-0.5,
  ymax=4,
  title={}]
  \addplot [only marks, mark size=2.5] table {
  0 0
  0 3
  1  1
  1  2
  2  0
  3  0
};
\addplot [only marks, mark=o, mark size=2.5] table {
 2  1
 1  0
 0  1
 0  2
};
\addplot [domain=0:1, samples=2]{-1*x+3};
\addplot [domain=1:3, samples=2]{-1*x+3};
\end{axis}

\end{tikzpicture}
\end{tabular}
\caption{Lattice points defining $\mathcal{L}$ and
$\widetilde{\mathcal{L}}$, respectively.}\label{transforming-figure}
\end{figure}
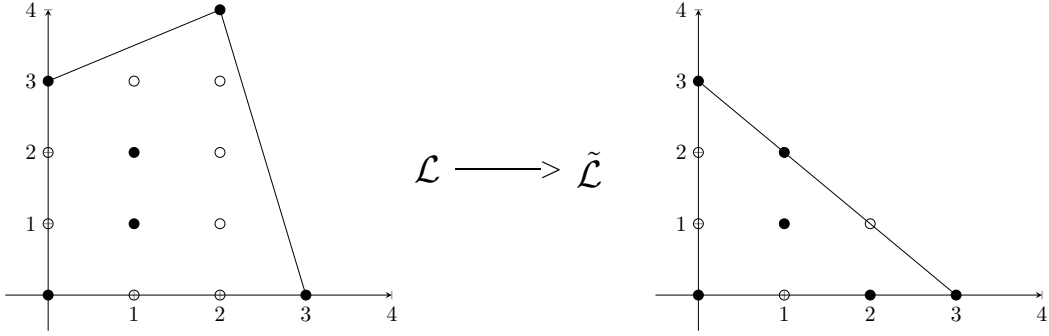
\end{example}

\begin{example}\label{5points-in-A2}
Let $K$ be the field $\mathbb{F}_3$, let
$X=\{(0,0),(1,0),(0,1),(1,1),(0,-1)\}$ be a set of $5$
points on $\mathbb{A}^2$, and let $I=I(X)$ be the vanishing ideal of $X$. The ideal
$I$ is generated by the binomials $t_1^2-t_1,t_2^3-t_2,t_1t_2^2-t_1t_2$. Using
Corollary~\ref{rth-GHW-affine} and Procedure~\ref{5points-in-A2-procedure}, we obtain 
$$
\begin{array}{lllll}
\delta_1(C_X(1))=2,&\delta_2(C_X(1))=4,&\delta_3(C_X(1))=5,& &\\
\delta_1(C_X(2))=1,&\delta_2(C_X(2))=2,&\delta_3(C_X(2))=3,&\delta_4(C_X(2))=4,&\delta_5(C_X(2))=5,
\end{array}
$$
$H_I^a(1)=\dim_K(C_X(1))=3$, and $H_I^a(2)=\dim(C_X(2)=5$.
\end{example}

\begin{example}\label{lex-example} 
Let $S=\mathbb{Q}[x,y]$ be a polynomial ring and let $\prec$ be the
lexicographical order with $x\succ y$. If $I=(g_1,g_2)$ is the ideal generated by
$g_1=y^{40}-y^2+1$ and $g_2=x-y^8$, and $f$ is the polynomial
$x^2-y^3+y$, then $f\in S_{\leq 3}$
but the residue of the division of $f$ by $\{g_1,g_2\}$ is
$h=y^{16}-y^3+y$ and $h\notin S_{\leq 3}$ (see
Procedure~\ref{lex-example-procedure}). This example shows that the
proof of Corollary~\ref{rth-GHW-affine} is only valid for graded orders.
\end{example}

\begin{example}\label{12points-A3-F3}
Let $K$ be the field $\mathbb{F}_3$ and let $X$ be the following
set of $12$ points in $\mathbb{A}^3$:
$$
\begin{array}{llllll}
\{(1,0,0), &(1,0,1), &(1,0,-1), &(1,1,0),&(1,1,1), &(1,1,-1),\\
\, \ (0,0,0), &(0,0,1),&(0,0,-1), &(0,1,0), &(0,1,1), &(0,1,-1)\}.
\end{array}
$$
\quad Using Corollary~\ref{min-dis-affine} and
Procedure~\ref{12points-A3-F3-procedure}, we obtain that ${\rm
reg}(H^a_{I(X)})=4$ and the minimum distance $\delta(C_X(d))$ of
$C_X(d)$ is given
by 
$$
\delta(C_X(1))=6,\ \delta(C_X(2))=3,\ \delta(C_X(3))=2,\
\delta(C_X(4))=1.
$$
\end{example}

\begin{example}\label{affine-torus-A2-F5}\rm Let $K$ be the field
$\mathbb{F}_{5}$, let $T$ be the affine torus
$(\mathbb{F}_5^*)^2$, and let $I=I(T)$ be the vanishing ideal of $T$. Using Theorem~\ref{rth-footprint-lower-bound}
and Procedure~\ref{affine-torus-A2-F5-procedure}, we obtain the
inequality ${\rm fp}_{I}(d,r)\leq\delta_r(C_T(d,r))$ and the
following table: 
\begin{eqnarray*}
\hspace{-11mm}&&\left.
\begin{array}{c|c|c|c|c|c|c}
d & 1 & 2 & 3 & 4 & 5 &6\\
   \hline
 H_I^a(d)    \    & 3 & 6 & 10&13 &15
 &16 \\ 
   \hline
{\rm fp}_{I}(d,1) &12&8& 4& 3 & 2&1 \\ 
 \hline
 {\rm fp}_{I}(d,2) & 15 & 11 & 7 & 4 & 3 & 2 \\
\hline
 {\rm fp}_{I}(d,3) & 16 & 12 & 8 & 6 & 4 & 3
\end{array}
\right.
\end{eqnarray*}
\end{example}

\begin{example}\label{10points-P3-F3}
Let $S=K[t_1,t_2,t_3]$ be a polynomial ring over the field
$K=\mathbb{F}_3$, let $X$ be the following 
set of $10$ points in $\mathbb{A}^3$:
$$
\begin{array}{l}
\{(1,0,0),\ (1,0,1),\ (1,0,-1),\ (1,1,0),\ (1,1,1),\ (1,1,-1),\ (0,0,1),\\
\ \,(0,1,0),\ (0,1,1),\ (0,1,-1)\},
\end{array}
$$
let $\mathcal{L}$ be the linear space $S_d$, and let $\mathcal{L}_X$
be the evaluation code on $X$. Using
Theorem~\ref{GHW-standard-code},
Proposition~\ref{transforming-to-standard-form}, and 
Procedure~\ref{10points-P3-F3-procedure}, we obtain the following
information:
$$
I(X)=(t_2^2-t_2,\,t_1^2-t_1,\,t_3^3-t_3,\,t_1t_2t_3-t_1t_2-t_1t_3-t_2t_3+t_1+t_2+t_3-1),
$$
if $d=2$ the linear space $\widetilde{\mathcal{L}}$ of
Proposition~\ref{transforming-to-standard-form} is generated by
$\{t_2,\,t_1t_2,\,t_1t_3,\,t_3^2,\,t_2t_3,\,t_1\}$,
the minimum distance $\delta(\mathcal{L}_X)$ of
$\mathcal{L}_X$ is given
by 
$$
\delta(\mathcal{L}_X)=6\mbox{ if }d=1,\ \delta(\mathcal{L}_X)=3\mbox{
if }d=2,\ \delta(\mathcal{L}_X)=1\mbox{ if }d=3,
$$
and $\delta_2(\mathcal{L}_X)=9+d-2$ for $d=2,3$.
\end{example}

\begin{example}\label{Hermitian-curve-example} 
Let $S=K[x,y]$ be a polynomial ring over the finite field
$K=\mathbb{F}_{25}$ with $25$ elements, let
$X=V_{\mathbb{A}^2}(g)$ be the 
Hermitian curve defined by $g=y^5+y-x^6$ \cite[p.~242]{tsfasman} and
let $I=I(X)$ be its vanishing ideal. If $X\neq\emptyset$, by 
Proposition~\ref{Nullstellensatz-finite}, the vanishing ideal of $X$ can be
computed using the equality
$$
I(X)=(x^{25}-x,\,y^{25}-y,\,g).
$$
\quad To compute the length of the Reed--Muller-type code $C_X(d)$
 of degree $d$ over the curve $X$, we use \textit{Macaulay}$2$ \cite{mac2} to obtain 
$$
|X|=|V_{\mathbb{A}^2}(g)|=\deg(S/I(X))
=\deg(\mathbb{F}_{25}[x,y]/(x^{25}-x,\,y^{25}-y,\,g))=125.
$$
\quad We can compute the minimum distance of $C_X(d)$ using
Proposition~\ref{Nullstellensatz-finite}, Corollary~\ref{min-dis-affine} and
Procedure~\ref{12points-A3-F3-procedure}. The $r$-th generalized Hamming
weight $\delta_r(C_X(d))$ and $r$-th footprint ${\rm
fp}_I(d,r)$ of $C_X(d)$ can be computed using 
Procedures~\ref{5points-in-A2-procedure} and \ref{affine-torus-A2-F5-procedure}. 
If $d=1$, we obtain a linear code $C_X(1)$ of length $125$,
dimension $3$, and minimum distance $119$. 
If $d=4$, using the footprint lower bound of
Corollary~\ref{rth-footprint-lower-bound-Reed-Muller}, we obtain a linear 
code $C_X(4)$ of length $125$, dimension $15$, minimum distance
at least $40$ because ${\rm fp}_I(4,1)=40$, and $\delta_7(C_X(4)\geq
97$ because ${\rm fp}_I(4,7)=97$. We can use this example as a model 
to estimate the parameters of $C_X(d)$ for any affine variety
$X=V_{\mathbb{A}^s}(G)$ 
defined by a given finite set $G$ of polynomials of $S$ in $s\geq 2$ 
variables, by replacing $g$ by
$G$.
\end{example}

\begin{example}\label{elliptic-curve-example} Let
$C=V_{\mathbb{A}^2}(f)\cup\{\mathcal{O}\}$ be the elliptic 
curve of the polynomial $f=y^2-x^3+x$ 
over a finite field $K=\mathbb{F}_q$ of ${\rm char}(K)\neq 2$. 
Elliptic curves have a group structure and they are 
used in cryptography \cite{Koblitz}. If $K=\mathbb{F}_{71}$, to
compute the number of zeros of 
$f$ in $\mathbb{A}^2$ notice that
$I(\mathbb{A}^2)=(x^{71}-x,y^{71}-y)$ \cite[p.~137]{JacI}. Then, using
Lemma~\ref{degree-formula-zeros-affine} and
\textit{Macaulay}$2$ \cite{mac2}, we obtain  
$$
|V_{\mathbb{A}^2}(f)|=\textrm{deg}(\mathbb{F}_{71}[x,y]/(I(\mathbb{A}^2),f))=71,
$$
see Procedure~\ref{elliptic-curve-procedure}. If we apply 
Theorem~\ref{degree-initial-footprint-affine} with the lexicographical order $y\succ x$, we obtain 
$$
|V_{\mathbb{A}^2}(f)|=\deg(S/(x^{71}-x,y^{71}-y,y^2-x^3+x))\leq\deg(S/(y^2,x^{71}))=142.
$$
\quad For a polynomial $g$ defining an elliptic curve over
a finite field $\mathbb{F}_q$, the bound of
Theorem~\ref{degree-initial-footprint-affine} says that
$|V_{\mathbb{A}^2}(g)|\leq 2q$, which is not a good bound 
(cf. Hasse's theorem \cite[p.~174]{Koblitz}). 

Let $X=V_{\mathbb{A}^2}(f)$ be the affine variety of $f$ and let
$I=I(X)$ be its vanishing ideal. We can
compute the minimum distance of the Reed--Muller-type code $C_X(d)$,
over the elliptic curve $X$, using
Proposition~\ref{Nullstellensatz-finite}, together with  
Corollary~\ref{min-dis-affine} and
Procedure~\ref{12points-A3-F3-procedure}. The $r$-th footprint of $C_X(d)$
can be computed using Procedure~\ref{affine-torus-A2-F5-procedure}. 
If $q=5$ and $d=1$, we obtain a linear code $C_X(1)$ of length $7$,
dimension $3$, minimum distance $4$, and the $1$-st footprint ${\rm
fp}_I(1,1)$ of $C_X(1)$ is $4$. 
If $q=71$ and $d=1$, we obtain a linear
code $C_X(1)$ of length $71$, dimension $3$, and minimum distance
$68$.  If $q=199$ and $d=10$, we obtain a
linear code $C_X(10)$ of length $199$, dimension $30$, and minimum
distance at least $57$ because ${\rm fp}_I(10,1)=57$. 
\end{example}

\begin{example}\label{Cox-example} 
Let $S=K[x,y,z]$ be a polynomial ring over the finite field
$K=\mathbb{F}_4$, let $\mathbb{P}^2$ be the projective space over the
field $K$, and let $\mathbb{X}=V_{\mathbb{P}^2}(g)$ be the
projective variety of the homogeneous polynomial 
$g=y^3+xz^2+x^2z$. Then, by
Theorem~\ref{Nullstellensatz-finite-projective}, the homogeneous 
vanishing ideal $I(\mathbb{X})$ of $\mathbb{X}$ is given by 
$$
I(\mathbb{X})={\rm rad}(xy^4-x^4y,\,xz^4-x^4z,\,yz^4-y^4z,\, g).
$$
\quad Computing the radical on the right with \textit{Magma} 
\cite{magma} (Procedure~\ref{Cox-procedure}) gives
\begin{equation}\label{may15-20}
I(\mathbb{X})=(x^4z+xz^4,\,x^2y+xyz+yz^2,\,x^2z+xz^2+y^3).
\end{equation}
\quad If $\mathcal{L}=S_1$, using Eq.~(\ref{may15-20}) and
\cite[Procedure~7.1, p.~334]{rth-footprint},  
we obtain that the standard evaluation code $\mathcal{L}_{\mathbb{X}}$ on  
$\mathbb{X}$ has length $9$, dimension $3$, and minimum distance $6$.
Note that we cannot apply Procedure~\ref{10points-P3-F3-procedure} since
$I(\mathbb{X})$ is not an affine vanishing ideal. 
\end{example}

\begin{example}\label{example-ptc} For $s=4$ and $q=3$, by
Theorem~\ref{Notre-Dame-Cathedral-Apr15-2019}, the list of values of the length,
dimension and minimum distance of $\mathcal{C}_\mathcal{P}(d)$ 
are given by the following table.
\begin{eqnarray*}
\hspace{-11mm}&&\left.
\begin{array}{c|c|c|c|c}
d & 1 & 2 & 3 & 4\\
   \hline
m & 16 & 16 & 16 & 16\\
   \hline
\dim_K(\mathcal{C}_\mathcal{P}(d)) & 4 & 6&  4& 1
 \\ 
   \hline
 \delta(\mathcal{C}_\mathcal{P}(d))   \    & 8 & 4  & 8&16\\
\end{array}
\right.
\end{eqnarray*}
\end{example}
\begin{appendix}

\section{Procedures for {\it Macaulay\/}$2$}\label{Appendix}

In this section we give 
procedures for \textit{Macaulay}$2$ \cite{mac2} and \textit{Magma}
\cite{magma} using finite fields to compute generalized
Hamming weights and lower bound footprints of evaluation codes.

\begin{procedure}\label{elliptic-curve-procedure} Computing the number of points
of an affine variety over a finite field using
Lemma~\ref{degree-formula-zeros-affine}.  This
procedure correspond to Example~\ref{elliptic-curve-example}. To
compute other examples just 
change the finite field and the set of polynomials $F$ that define the affine variety.  
\begin{verbatim}
q=71
S=ZZ/q[x,y]-- finite field K=ZZ/q 
I=ideal(x^q-x,y^q-y)--vanishing ideal of K^2 
F={y^2-x^3+x}
quotient(I,ideal(F))==I--false means F has zeros
degree (I+ideal(F))--number of zeros of f
\end{verbatim}
\end{procedure}

\begin{procedure}\label{transforming-procedure} 
Computing the generalized Hamming weights and footprint of an
evaluation code  
$\mathcal{L}_X$ using Theorem~\ref{GHW-standard-code}, 
Proposition~\ref{transforming-to-standard-form}, and
Theorem~\ref{rth-footprint-lower-bound}.  
The input for this
procedure is a generating set for $\mathcal{L}$ and the vanishing 
ideal $I$ of $X$. This procedure 
corresponds to Example~\ref{transforming-example}.
\begin{verbatim}
q=5, H=GF(q,Variable=>a), S=H[t1,t2]
I=ideal(t1^(q-1)-1,t2^(q-1)-1), M=coker gens gb I
r=1--we are computing the r-th generalized Hamming weight 
G=gb I, div=(x)->x % G
--This is the K-basis for L
Basis=matrix{{1,t1^3,t1*t2^2,t2^3,t1*t2,t1^2*t2^4}}
--This is the list of remainders on division by G
cL=toList set apply(flatten entries Basis,div)
--This gives the r-th generalized Hamming weight
gmd=degree M-max apply(apply(subsets(toList set apply(toList 
set(apply(apply(apply(apply(toList 
((set(0,a,a^2,a^3,a^4))^**(#cL)-(set{0})^**(#cL))/deepSplice,
toList),x->matrix{cL}*vector x),entries),n->n#0)),
m->(leadCoefficient(m))^(-1)*m),r),ideal),
x-> if #(set flatten entries leadTerm gens x)==r 
then degree(I+x) else 0)
init=ideal(leadTerm gens gb I)
er=(x)-> degree ideal(init,x)
--This is the r-th footprint
fpr=degree M - max apply(apply(apply(subsets(cL,r), 
toSequence),ideal),er) 
\end{verbatim}
\end{procedure}

\begin{procedure}\label{5points-in-A2-procedure} Computing the
generalized Hamming weights of an affine Reed--Muller-type code using
Corollary~\ref{rth-GHW-affine}. This
procedure correspond to Example~\ref{5points-in-A2}. Other examples
can be computed changing the finite field, the affine space, and the
set of points of $X$.   
\begin{verbatim}
q=3, G=ZZ/q, S=G[t1,t2];
I1=ideal(t1,t2), I2=ideal(t1-1,t2),I3=ideal(t1,t2-1)
I4=ideal(t1-1,t2-1),I5=ideal(t1,t2+1)
I=intersect(I1,I2,I3,I4,I5)--this is the vanishing ideal
M=coker gens gb I
--This is the r-th generalized Hamming weight of C_X(d):
genmdaffine=(d,r)->degree M-max apply(apply(subsets(toList 
set apply(toList set(apply(apply(apply(apply(toList 
((set(0..q-1))^**(#flatten entries basis(0,d,M))-
(set{0})^**(#flatten entries basis(0,d,M)))
/deepSplice,toList),x->basis(0,d,M)*vector x),
entries),n->n#0)), m->(leadCoefficient(m))^(-1)*m),r),
ideal), x-> if #(set flatten entries leadTerm gens x)==r 
then degree(I+x) else 0)
--This is the affine Hilbert function of I:
#flatten entries basis(0,1,M), #flatten entries basis(0,2,M)
genmdaffine(1,1), genmdaffine(1,2),genmdaffine(1,3)
genmdaffine(2,1), genmdaffine(2,2)
\end{verbatim}
\end{procedure}

\begin{procedure}\label{lex-example-procedure}
Computing the quotient and the remainder in the multivariate division
algorithm \cite[Theorem~3, p.~63]{CLO}. This procedure correspond to Example~\ref{lex-example}
\begin{verbatim}
R=QQ[x,y,MonomialOrder=>Lex]
I=ideal(y^(40)-y^2+1,x-y^8)
f=x^2-y^3+y
G=matrix{{y^(40)-y^2+1,x-y^8}}
remainder(matrix{{f}},G)
quotientRemainder(matrix{{f}},G)
\end{verbatim}
\end{procedure}

\begin{procedure}\label{12points-A3-F3-procedure}
Computing the minimum distance of an affine Reed--Muller-type code of
degree $d$ using Corollary~\ref{min-dis-affine}. This procedure
corresponds to Example~\ref{12points-A3-F3} 
\begin{verbatim}
q=3, S=ZZ/q[t1,t2,t3];
I1=ideal(t1-1,t2,t3), I2=ideal(t1-1,t2,t3-1),I3=ideal(t1-1,t2,t3+1),
I4=ideal(t1-1,t2-1,t3),I5=ideal(t1-1,t2-1,t3-1),I6=ideal(t1-1,t2-1,t3+1),
I7=ideal(t1,t2,t3), I8=ideal(t1,t2,t3-1),I9=ideal(t1,t2,t3+1),
I10=ideal(t1,t2-1,t3),I11=ideal(t1,t2-1,t3-1),I12=ideal(t1,t2-1,t3+1)
I=intersect(I1,I2,I3,I4,I5,I6,I7,I8,I9,I10,I11,I12)
M=coker gens gb I
--This computes the minimum distance of an affine Reed-Muller-type code
--of degree d
mindisaffine=(d)-> degree M- max apply(apply((toList (set
(apply(toList set apply(toList set(apply(apply(apply(apply(toList 
((set(0..q-1))^**(#flatten entries basis(0,d,M))-
(set{0})^**(#flatten entries basis(0,d,M)))/deepSplice,toList),
x->basis(0,d,M)*vector x),entries),n->n#0)),
m->(leadCoefficient(m))^(-1)*m),x-> if degree(x)=={d} then x 
else t1^0))-set{t1^0})),ideal),x-> degree(I+x))
mindisaffine(1), mindisaffine(2), mindisaffine(3)
\end{verbatim}
\end{procedure}

\begin{procedure}\label{affine-torus-A2-F5-procedure} Computing the
footprint of a Reed--Muller-type code $C_X(d)$. This procedure corresponds to
Example~\ref{affine-torus-A2-F5}. To compute other examples just
change the finite field and the vanishing ideal $I$ of $X$. 
\begin{verbatim}
q=5, G=ZZ/q, S=G[t1,t2];
I=ideal(t1^(q-1)-1,t2^(q-1)-1)--vanishing ideal of the affine torus T
M=coker gens gb I
init=ideal(leadTerm gens gb I)
er=(x)-> if not quotient(init,x)==init then degree ideal(init,x) else 0
--This is the r-th footprint:
fpraffine=(d,r)->degree M - max apply(apply(apply(subsets
(flatten entries basis(0,d,M),r),toSequence),ideal),er)
f=(n)->#flatten entries basis(0,n,M), apply(1..6,f)
f1=(n)->fpraffine(n,1),apply(1..6,f1)
f2=(n)->fpraffine(n,2),apply(1..6,f2)
f3=(n)->fpraffine(n,3),apply(1..6,f3)
\end{verbatim}
\end{procedure}

\begin{procedure}\label{Cox-procedure} Computing the radical of an
ideal over a finite 
field using \textit{Magma} \cite{magma}. This procedure corresponds to
Example~\ref{Cox-example}.
\begin{verbatim}
P<x,y,z>:=PolynomialRing(FiniteField(2, 2),3);
J:= ideal<P| x*y^4-x^4*y,x*z^4-x^4*z,y*z^4-y^4*z,y^3+x*z^2+x^2*z>;
Radical(J);
\end{verbatim}
\end{procedure}

\begin{procedure}\label{10points-P3-F3-procedure}
Given an evaluation code $\mathcal{L}_X$ on $X$ and a monomial order
$\prec$. This procedure computes a linear subspace
$\widetilde{\mathcal{L}}$ of $S$ such that 
$\widetilde{\mathcal{L}}_X$ is a standard evaluation code on $X$ 
and $\widetilde{\mathcal{L}}_X=\mathcal{L}_X$ 
(Proposition~\ref{transforming-to-standard-form}). Then it computes
the $r$-generalized 
Hamming weight of a projective 
Reed--Muller-type code on $X$ of degree $d$
(Theorem~\ref{GHW-standard-code}). This procedure corresponds to 
Example~\ref{10points-P3-F3}. To compute $\delta_r(\mathcal{L}_X)$ for
an evaluation code $\mathcal{L}_X$ replace \texttt{basis(d,S)} by the
matrix of a $K$-basis of the linear space $\mathcal{L}$.
\begin{verbatim}
q=3, S=ZZ/q[t1,t2,t3];
I1=ideal(t1-1,t2,t3),I2=ideal(t1-1,t2,t3-1),I3=ideal(t1-1,t2,t3+1),
I4=ideal(t1-1,t2-1,t3),I5=ideal(t1-1,t2-1,t3-1),
I6=ideal(t1-1,t2-1,t3+1),I7=ideal(t1,t2,t3-1),I8=ideal(t1,t2-1,t3),
I9=ideal(t1,t2-1,t3-1),I10=ideal(t1,t2-1,t3+1)
I=intersect(I1,I2,I3,I4,I5,I6,I7,I8,I9,I10)
M=coker gens gb I
d=2, r=1, G=gb I
div=(x)->x % G
--This is the list of residues of S_d after division by G
cL=toList set apply(flatten entries basis(d,S),div)
--This gives the r-th generalized Hamming weight of 
--the evaluation code on X defined by S_d
gmd=degree M-max apply(apply(subsets(toList set apply(toList 
set(apply(apply(apply(apply(toList ((set(0..q-1))^**(#cL)-
(set{0})^**(#cL))/deepSplice,toList),x->matrix{cL}*vector x),
entries),n->n#0)),m->(leadCoefficient(m))^(-1)*m),r),ideal),
x-> if #(set flatten entries leadTerm gens x)==r 
then degree(I+x) else 0)
\end{verbatim}
\end{procedure}
\end{appendix}


\section*{Acknowledgments} 
We thank Nupur Patanker for pointing out an error in the previous 
statement of Corollary~\ref{case-s-2d} and
Proposition~\ref{squarefree-affine}. 
Computations with \textit{Magma} \cite{magma} and 
\textit{Macaulay}$2$ \cite{mac2} were important to 
give examples and to have a better understanding of evaluation
codes. 

\bibliographystyle{plain}

\end{document}